\documentclass[12pt]{amsart}

\usepackage{fullpage}
\usepackage{xcolor}
\usepackage{graphicx}
\usepackage{amssymb, amsmath, amsthm}
\usepackage{comment}

\usepackage{thmtools, thm-restate}

\usepackage{hyperref}

\newcommand{\McC}{$^{\text{c}}$}

\newcommand{\uhp}{\mathbb{H}}
\newcommand{\R}{\mathbb{R}}
\newcommand{\C}{\mathbb{C}}
\newcommand{\B}{\mathbb{B}}
\newcommand{\D}{\mathbb{D}}

\newcommand{\sgn}{\text{sgn}}

\newtheorem{theorem}{Theorem}[section]
\newtheorem{prop}[theorem]{Proposition}
\newtheorem{lemma}[theorem]{Lemma}
\newtheorem{corollary}[theorem]{Corollary}

\theoremstyle{remark}
\newtheorem{example}[theorem]{Example}
\newtheorem{remark}[theorem]{Remark}
\newtheorem{question}[theorem]{Question}

\title{Stable polynomials and bounded rational functions in the unit ball}
\author{Greg Knese}
\address{Department of Mathematics, Washington University in St. Louis, St. Louis, MO 63130, USA}
\email{geknese@wustl.edu}

\author{James Eldred Pascoe}
\address{Department of Mathematics, Drexel University, Philadelphia, PA 19104, USA}
\email{jep362@drexel.edu}

\author{Alan Sola}
\address{Department of Mathematics, Stockholm University, 106 91 Stockholm, Sweden}
\email{sola@math.su.se}
\date{\today}                                           
\thanks{GK partially supported by NSF grant DMS-2247702}

\begin{document}

\begin{abstract}
We study polynomials with no zeros on the
unit ball in complex Euclidean space 
with a view toward characterizing
when a rational function is bounded on
the ball.  
We give a complete local description
of such polynomials in two variables near a boundary zero.
In higher dimensions, we give a partial characterization 
of a simple boundary zero.
Several applications are given including boundedness of rational functions
with boundary singularities and
constructions of examples with prescribed local properties.
\end{abstract}

\maketitle


\section{Introduction}

Stable polynomials, those with no zeros in some domain $\Omega \subset \C^d$,
are fundamental
in complex analysis---most visibly as the denominators of
rational functions analytic on $\Omega$.  
While such polynomials have been studied a great deal when $\Omega$
is a standard product domain such as the unit polydisk $\D^d$
or the product of upper half planes, they have been much
less studied in the setting of one of the most common
model domains in several complex variables, the Euclidean unit
ball in $\C^d$, 
\[
\B_d = \left\{z=(z_1,\dots,z_d) \in \C^d: |z|^2 = \sum_{j=1}^{d}|z_j|^2<1\right\}.
\]
The reasons for this discrepancy are varied but perhaps
rely on the facts that (1) various applications necessarily lead
to product domains due to some sort of independence of the variables
involved and (2) product domains have a variety of useful symmetries
(for example, reflection across the unit circle)
that make them more tractable to study.  
The ball of course lacks product structure and the reflective symmetries
of the polydisk but still possesses a transitive automorphism group \cite[Chapter 2]{Rud}.
Polynomials with no zeros in $\B_d$ have certainly appeared in the literature
especially in the context of function spaces on the ball and in particular the
study of cyclicity of polynomials (with respect to shift operators), see eg. \cite{S,KV,APRS}.
They also appear in the study of proper rational maps between balls (see \cite{dA2})
as well as in the study of interpolation problems for bounded analytic functions
on the ball (see \cite{KZ}).

The purpose of this article is to initiate the study of polynomials with no zeros 
on the ball for its own sake while also keeping in mind applications to 
understanding analytic functions on the ball.  In particular, a basic 
question, called the admissible numerator problem, is to characterize bounded rational functions $f= q/p$ on the ball
when the denominator $p$ has a boundary zero which without loss of generality 
can be assumed to be at $(0,\ldots,0, 1)$\footnote{A unitary transformation can take
any boundary point to this one.}. 
In other words, we seek to describe the ideal 
\[
\mathcal{I}_p^{\infty}=\{q\in \mathbb{C}\{z_1,\ldots, z_d\}\colon q/p \, \, \textrm{is bounded on \,\,} \mathbb{B}_d \cap B((0,\ldots,0, 1), \epsilon) \, \textrm{for some}\, \epsilon>0\}.
\]
Here $\C\{z_1,\dots,z_d\}$ is the ring of convergent power series in the
given variables centered at $0$.

Since boundary zeros of $q$ and $p$ can compete without canceling, this is a subtle question: it is generally not enough for $q$ to vanish at boundary zeros of $p$. 

\begin{example}
Consider the polynomial $p(z_1,z_2)=1-z_2$ which has a single zero in $\overline{\mathbb{B}_2}$ at the boundary point $(0,1)$. For this simple example, we can determine the generators of the ideal $\mathcal{I}^{\infty}_{p}$ by elementary means. First note that
\[|z_1|^2\leq 1-|z_2|^2=(1+|z_2|)(1-|z_2|)\leq 2(1-|z_2|), \quad (z_1, z_2)\in \mathbb{B}_2.\] 
Hence, the rational function $f(z)=z_1^2/(1-z_2)$ is bounded in $\mathbb{B}_2$. Trivially, $1-z_2 \in \mathcal{I}^{\infty}_p$ and hence $(z_1^2,1-z_2)\subset \mathcal{I}^{\infty}_p$. To prove that $(z_1^2, 1-z_2)=\mathcal{I}^{\infty}_p$ it suffices to show that $1,z_1\notin \mathcal{I}^{\infty}_p$.  

It is clear that $1/(1-z_2)$ is not bounded in the ball. Let us show that $f(z)=z_1/(1-z_2)$ is unbounded also. 
First, note that for fixed $0<a<1$,
\[r^2+a^2(1-r^2)< r^2+1-r^2=1, \quad r\in [0,1),\]
and hence $\gamma=\{(a(1-r^2)^{1/2},r)\colon r\in [0,1)\}$
is a curve in $\mathbb{B}_2$ terminating at $(0,1)$. Now
\[\lim_{r\to 1}|f(a(1-r^2)^{1/2}, r)|^2=\lim_{r\to 1}\frac{a^2(1-r^2)}{(1-r)^2}
=\lim_{r\to 1}\frac{a^2(1+r)}{1-r}=\infty. \quad \diamondsuit
\]
\end{example}

The analogous problem of characterizing admissible numerators for the bidisk $\D^2$ was answered in \cite{BKPS},\cite{Kollar} (see also \cite{Kprep}) and the ``generic'' situation was addressed in the paper \cite{BKPS2} for 
the polydisk $\D^d$.  
In these cases, and in the present paper, it was necessary to understand
local properties of the zero set of a polynomial without zeros
on the domain near a boundary zero.
In two variables and in the setting of the bidisk, a complete
description is possible; see \cite{BKPS}, pg 98, ``Theorem (Puiseux Factorizations)''.
For higher dimensional polydisks,
there is still much to be discovered but certain special
cases such as when $p$ has an isolated boundary zero 
at which the zero set of $p$ is smooth can be addressed; 
see \cite{BKPS2}, Theorem 1.10.

When studying the local properties of the 
zero set of a ball-stable polynomial
at a boundary zero 
it is convenient to convert from 
the unit ball $\mathbb{B}_d$ to the biholomorphically equivalent
 Siegel half-space
\[
U_d = \left\{(z_1,\dots,z_{d-1},w)\in \C^d: \Im w > \sum_{j=1}^{d-1} |z_j|^2\right\}
\]
via a rational biholomorphic map.
The most basic description and situation occurs when $p$
has a simple zero at $0$.

\begin{restatable}{theorem}{smooththm}\label{smooththm}
Let $p \in \C[z_1,\dots, z_{d-1},w]$ be non-vanishing
in $U_d$ near $0$, $p(0) = 0$, and $\nabla p(0) \ne 0$.
Let $Hp(0) = \left(\frac{\partial^2 p}{\partial z_j z_k}(0)\right)_{j,k=1,\dots,d-1}$ be the Hessian matrix of $p$ with respect to $z$ at $0$.
Then, $\nabla p(0)$ is a multiple of $\vec{e}_d$ and
\[
\frac{1}{2|\nabla p(0)|} Hp(0)
\]
is contractive.

For a partial converse, if $p$ is a polynomial
such that $p(0)=0$, $\nabla p(0)$ is a multiple of $\vec{e}_d$,
and $\frac{1}{2|\nabla p(0)|} Hp(0)$ is strictly contractive,
then $p$ is non-vanishing near $0$ within $U_d$.
Furthermore, in this case
for any function $q$ analytic near $0$, we have that 
$q/p$ is bounded near $0$ within $U_d$  
if and only if $q \in (w, (z)^2)$.
\end{restatable}

Here $(z)^2$ is the ideal generated by $z^{\alpha}$ with $|\alpha| = 2$
and $(w,(z)^2)$ is the ideal generated by $w$ and $(z)^2$.
The ambient ring can be taken to be the ring of convergent power
series centered at $0$, denoted $\C\{z_1,\dots, z_{d-1},w\}$.
The case where $\frac{1}{2|\nabla p(0)|} Hp(0)$ is strictly
contractive ends up implying especially simple behavior
as $p(z_1,\dots,z_{d-1},w)$ is comparable above and below to $w$
within $U_d$.
The case where $\frac{1}{2|\nabla p(0)|} Hp(0)$ has operator
norm $1$ is already subtle in the two variable case.

In two variables, we can move beyond a simple boundary zero
and beyond the strict contractivity condition above. By the Newton-Puiseux theorem, 
an algebraic curve that avoids $U_2$ will have branches parametrized locally by
\(
t\mapsto (t^M, \phi(t)) 
\)
where $\phi$ is analytic and vanishing at $0$ and satisfies
\[
 |t|^{2M}\geq \Im(\phi(t)) 
\]
for $t\in \C$ sufficiently close to $0$.
We can perform a rotation of our parameter and assume the parametrization is of the form
\(
t \mapsto (c t^M, \phi(t))
\)
where $|c|=1$ and the first non-zero coefficient of $\phi$ lies on the positive imaginary axis.
This is just a normalization that simplifies some formulas later. 

Now let us write
\[
\phi(t) = \sum_{j=1}^{\infty} \phi_j t^j = \psi_0 t^{a_0} + o(t^{a_0})
\]
where the first sum is the standard power series expansion for $\phi$
while the second expression records the first non-zero term that appears
(here $\psi_0 \in i(0,\infty)$ and $a_0\in \mathbb{N}$).

Then, our requirement on $\phi$ is
\begin{equation} \label{requirement}
 G_{\phi}(t) := |t|^{2M} - \Im (\phi(t) )\geq 0.
\end{equation}

Our goals are to:
\begin{itemize}
    \item describe all such $\phi$ 
    \item see how properties of $\phi$
can be applied to study admissible numerators of polynomials non-vanishing
in $U_2$ near $(0,0)$
    \item construct polynomials with prescribed local and/or global behavior.
\end{itemize}

Just as in Theorem \ref{smooththm}, we have the following
result.

\begin{restatable}{prop}{basictwovar} \label{basictwovar}
Let $\phi(t) = \psi_0 t^{a_0} + o(t^{a_0}) \in \C\{t\}$ as above with $\psi_0 \in i(0,\infty)$.
If \eqref{requirement} holds for $|t|$ small, then $a_0 \geq 2M$.
If in addition $a_0=2M$, then $|\psi_0| \leq 1$.

Conversely, if $a_0 > 2M$ or $a_0=2M$ and $|\psi_0| < 1$, then \eqref{requirement}
holds for $|t|$ small and we have $w - \phi(z^{1/M})$ comparable
to $w$.
\end{restatable}

Thus, again we are led to simple behavior in one case, $|\psi_0|<1$,
and more possibilities in the borderline case of $\psi_0 = i$ and $a_0=2M$.
In the case $|\psi_0|<1$, the expression $w-\phi(z^{1/M})$ has an isolated zero
on $\partial U_2$ 
while $\psi_0 = i$ opens up other possibilities
including that $w - \phi(z^{1/M})$ could have a
real 1-dimensional curve of zeros on $\partial U_2$.
Curiously, the only previously known examples in the literature 
of irreducible polynomials
with no zeros in $U_2$ and with a curve of zeros on $\partial U_2$ 
were $w-z^2$ and closely related modifications.
In the ball setting this corresponds to $1-z_1^2-z_2^2$ 
or $1-2z_1z_2$ (and the modifications would correspond
to unitary change of variables).
With the present paper we can add the following example:
\begin{equation}
w+w^2 -z^2
\label{newexample}
\end{equation}
which again has no zeros in $U_2$ 
but a curve of zeros on the boundary. A version of this example in the setting of $\mathbb{B}_2$ is
\[i-1+2z_2-(1+i)z_2^2+z_1^2.\]
The boundary curve in $\partial U_2$ is parametrized by
\[
\theta \mapsto \left(\pm \frac{\sqrt{2}}{2} e^{i\theta/2} e^{i\pi/4} \sqrt{\sin \theta}, -\frac{1}{2} + \frac{1}{2} e^{i\theta}\right), \quad  \theta\in [0,\pi].
\]  
The first component plots the lemniscate of Bernoulli and the 
second component is a half-circle.  See Figure \ref{fig:Bernoulli}
and Example \ref{ex:exotic}.

\begin{figure} \label{fig:Bernoulli}
\includegraphics[scale = 0.6]{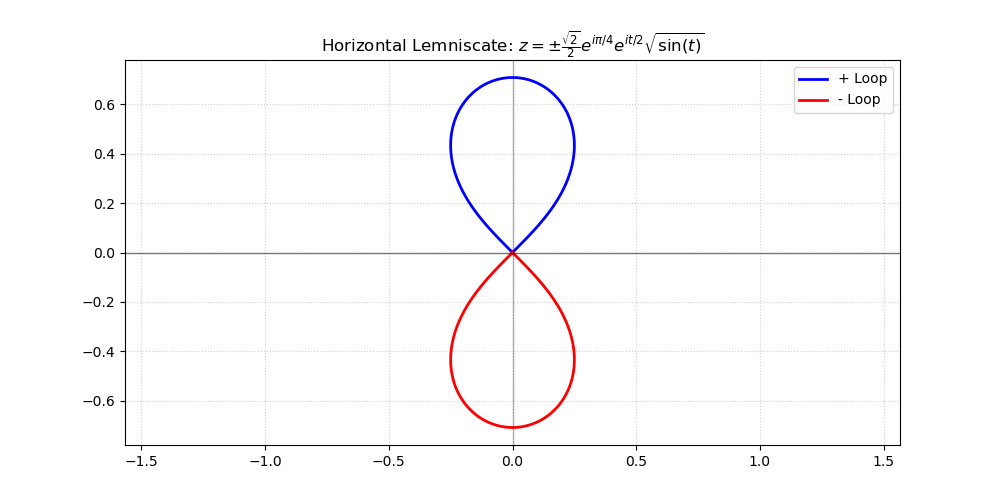}
\caption{Bernoulli Lemniscate}
\end{figure}

What other examples are there?

\begin{question}
Is it possible to characterize irreducible $p \in \C[z_1,z_2]$ 
with no zeros in $\B_2$ or $U_2$ and infinitely many boundary zeros?
\end{question}

Our methods answer a local analytic version of this question.
The easiest way to describe a local analytic example is to give a parametrization
of its zero set.
The analytic curve
\begin{equation}
s\mapsto (s e^{is}, is^2 - (2/3) s^3)
\label{M1example}
\end{equation}
defined for small $s\in \C$
avoids $U_2$ and maps real numbers to $\partial U_2$.
To get an analytic function with the desired
local properties we let $\Phi(z)$ be
the inverse function of $s\mapsto z=se^{is}$
and then $w - i\Phi(z)^2 - (2/3)\Phi(z)^3$
has no zeros in $U_2$ near $(0,0)$
yet has a curve of zeros on $\partial U_2$.
See Example \ref{excurve-u-new} for details.
In two variables, 
polynomials with a curve of zeros on the boundary
are not relevant to the study of
bounded rational functions on 
the ball because such polynomials
do not occur as denominators of bounded
$q/p$ written with $q$ and $p$ having
no common factors.\footnote{Specifically,
in order for $q/p$ to be bounded 
it is necessary that $Z_p\cap \partial \mathbb{B}_2 \subset Z_q\cap \partial \mathbb{B}_2$
and infinitely many common zeros of $p$ and $q$ imply
a common factor by B\'ezout's theorem.}
The phenomenon is interesting nonetheless.

The next theorem gives a complete characterization
of the relevant $\phi$ of the form $\phi(t) = it^{2M} + o(t^{2M})$.
\footnote{It is remarkable that some of the structure present
mirrors what happens in the setting of the bidisk or bi-upper half-plane.
See \cite{BKPS}, pg 98, ``Theorem (Puiseux Factorizations).''}
Taken together, Proposition \ref{basictwovar} and Theorem \ref{mainthm}
give a complete local description of polynomials
with no zeros in $U_2$ near $(0,0)$.

\begin{restatable}{theorem}{mainthm}\label{mainthm}
Given $M\in \mathbb{N}$, $|c|=1$
and $\phi \in \C\{t\}$, 
$\phi(t) = it^{2M} + o(t^{2M})$ such that
\[
t \mapsto (ct^M, \phi(t))
\]
is an injective parametrization
of an analytic curve that avoids $U_2$ for $t\in \C$ small, there exists $L(s) \in \C\{s\}$ with $L(0)=0$
such that our curve can be parametrized by
\begin{equation}\label{Lparam}
s\mapsto \left(cs^M e^{M L(s)}, is^{2M} 
+ 2Mi\int_{[0,s]} w^{2M} L'(w) dw\right).
\end{equation}
Furthermore, the curve falls into
two distinct types:
\begin{enumerate}
    \item (curve type) $M=1$ and $L(s) \in i\R\{s\}$
    in which case our curve intersects
    $\partial U_2$ along a real one dimensional curve.
    \item (isolated type) $M\geq 1$ and 
    $L(s)$ has the form
\[
L(s) = iA_0(s^M) + s^{2MK}L_1(s)
\]
where $K\in \mathbb{N}$, 
$A_0(s) \in \R[s]$ has degree less than $2K$,
and $L_1 \in \C\{s\}$, $\Re(L_1(0))> 0$.
    In this case, our curve intersects
    $\partial U_2$ at the isolated point $(0,0)$.
    \end{enumerate}
Conversely, given $M$ and $L(s)$ 
satisfying either of the two conditions above,
the parametrization \eqref{Lparam}
yields an analytic curve
that avoids $U_2$ near $(0,0)$.    
For the parametrization to be injective 
we need $L$ to not be a function of $s^M$.
\end{restatable}

The parametrization above is initially off-putting
but really only involves some basic conditions
on a power series $L$ and the computation of an 
integral (which just operates as a multiplier on 
the coefficients).
Even more simply, if $H(s) = s^{M} e^{ML(s)}$,
then
\[
is^{2M} 
+ 2Mi\int_{[0,s]} w^{2M} L'(w) dw = 2i\int_{[0,s]} w^{2M} \frac{H'(w)}{H(w)}dw.
\]
In \eqref{M1example} above we wrote out the case $M=1, L(s)=is$. The simplest curve type example is $M=1$, $L(s)\equiv 0$ which corresponds to the 
globally stable polynomial $w-z^2$ but as \eqref{newexample} demonstrates, there are other $L$ with this property.
The simplest example for $M=2$ (which is necessarily
of isolated type) is $L(s) = is^2+s^4+s^5$.
See Example \ref{M2ex}.

We can also describe the admissible numerator
ideal for polynomials with zero set of the isolated type above.
In this setting the admissible numerator ideal
is defined
\[
\mathcal{I}_p^{\infty}=\{q\in \mathbb{C}\{z,w\}\colon q/p \, \, \textrm{is bounded on \,\,} U_2 \cap B((0,0), \epsilon) \, \textrm{for some}\, \epsilon>0\}.
\]
One consequence of the description above is that in
the isolated type case, $\phi$ has the form  
\begin{equation}\label{phi0}
\phi(t) = \phi_0(t^M) + O(t^{2M(1+K)})
\end{equation}
where $\phi_0(t) \in \C[t]$ has degree less than $1+K$.
\footnote{It is tempting to try to characterize $\phi$ directly
based on this but it does not seem so straightforward.}
See Corollary \ref{cor:compare}.

\begin{restatable}{theorem}{admissiblenonbasic} \label{admissiblenonbasic}
    Let $p(z,w)$
    be a polynomial 
    with no zeros in $U_2$ near $(0,0)$
    whose zero set near $(0,0)$
    is (injectively) parametrized by 
    \[
    t \mapsto (ct^M, \phi(t))
    \]
    where $\phi$ is of the isolated type as in Theorem \ref{mainthm}.
    Then, the admissible numerator ideal for $p$
    is given by 
    \begin{equation} \label{productideal}
    (w-\phi_0(\bar{c} z), z^{2(1+K)})^{M}
    \end{equation}
    where we refer to the data $K, \phi_0$ in Theorem \ref{mainthm}
    and \eqref{phi0}.
\end{restatable}

The expression \eqref{productideal} is a product ideal:
combinations of $(w-\phi_0(\bar{c} z))^j z^{2(1+K)(M-j)}$
for $j=0,\dots, M$ in $\C\{z,w\}$.

Theorem \ref{admissiblenonbasic} focuses
on when $p$ has a single irreducible Weierstrass
factor around $(0,0)$ with an isolated type
parametrization.
What about determining admissible numerators more 
generally in two variables?
In general, a polynomial with no zeros in $U_2$ near
$(0,0)$ could factor into several irreducible Weierstrass
polynomials and each would have its zero set parametrized
by a Puiseux type of parametrization---and these could
each be either basic, isolated, or curve types. 
We leave it for future work to determine the admissible
numerator ideal.  
In the bidisk setting, this is the
step originally established in Koll\'ar \cite{Kollar}.
(One difference at a purely algebraic level
is that the admissible numerator ideal
in the bidisk setting involved real polynomials,
and this played some role in the proof in \cite{Kollar}.
Here the analogue, $\phi_0$, may not be real or purely imaginary.)

The next big question is if one can construct
globally defined polynomials with the behaviors
described in our local theory.
Section \ref{examplesection} goes through 
several examples of polynomials 
exhibiting various aspects of the local theory above.
A construction of Rudin from \cite{Rud}
makes it possible to construct global examples 
of polynomials with parametrizations satisfying
$M=1$ and $K$ arbitrarily large.
See Section \ref{Rudin}.

Determinantal formulas are valuable in
the polydisk (and especially bidisk) setting,
but not apparently in the ball setting.
We point out in Section \ref{sec:row}
that a natural class of determinantal polynomials to consider
in the ball setting, while non-vanishing on $\B_d$,
have particularly simple boundary behavior.
We have in mind the polynomials
\[
\det\left(I - \sum_{j=1}^{d} z_j A_j\right)
\]
where $(A_1,\dots, A_d)$ is a row contraction of
$N\times N$ matrices (meaning $\sum_{j=1}^{d} A_jA_j^* \leq I_N$),
which turn out to have finitely many boundary zeros,
each of which contributes a degree \emph{one} factor
of the polynomial.  See Proposition \ref{factordet}.

We conclude the paper with Section \ref{sec:global}
which gives a method to construct polynomials
with the correct local behavior however
we do not see how to enforce that the constructed polynomials
have no zeros in $U_2$.

Suppose we have an irreducible polynomial $p\in \C[z,w]$
with no zeros in $U_2$ and a zero at $(0,0)$.
By Proposition \ref{basictwovar}, 
$p$ may have branches of the basic types:
\[
s\mapsto (c s^M, \psi_0 i t^{a_0} + o(t^{a_0}))
\]
where either $a_0 > 2M$ or ``$a_0=2M$ and $\psi_0<1$.''
It is easy to construct polynomials
with this local behavior for arbitrary $M$:
\[
s\mapsto (s^M, is^{2M+1}) \quad \text{ and } \quad s\mapsto (s^M, (i/2)s^{2M}+s^{2M+1})
\]
are parametrizations with this behavior and
the corresponding polynomials with zero sets 
given by the ranges of these are
\[
z^{2M+1}+w^{2M} \text{ and } (w-(i/2)z^2)^{M}-z^{2M+1}.
\]
However, these are not non-vanishing throughout $U_2$.
Nevertheless, Examples \ref{Nextsimplest} and \ref{Deg4} exhibit
polynomials with the correct
local and global behavior for the basic type 
(as well as one instance of \emph{isolated type}), 
so we would really like
to find a way to construct polynomials 
non-vanishing on $U_2$ with boundary zeros
of isolated or curve types.

So, let's suppose $p$ has at least one branch
with local parametrization
\begin{equation} \label{gamma}
s \mapsto \gamma(s) = \left(cs^Me^{M L(s)}, is^{2M} +i2M \int_{[0,s]} x^{2M} L'(x)dx\right)
\end{equation}
where $L(s) \in \C\{s\}$, $L(0)=0$, $|c|=1$, $M \in \mathbb{N}$ with additional details 
laid out in Theorem \ref{mainthm} depending
on isolated type versus curve type.

\begin{restatable}{theorem}{algthm} \label{algthm}
If the above parametrization \eqref{gamma} is algebraic,
then $f(s) = e^{L(s)}$ is an algebraic function
whose analytic continuation has no
zeros or poles (nor algebraic zeros/poles)
over points of $\C_{*} = \C\setminus \{0\}$.
In addition, the Puiseux-Laurent series of (any analytic
continuation) of $f'(1/s)/f(1/s)$ around $0$ cannot contain the term
$s^{2M+1}$.  

\end{restatable}

By algebraic zero or pole, we mean
that at a branch point $s_0$ of the continuation of
$e^{L}$
the function necessarily has a Puiseux-Laurent
development
\[
\sum_{j=n}^{\infty} a_j(s-s_0)^{j/k} \qquad a_n\ne 0
\]
and a zero would correspond to $n>0$ 
while a pole would correspond to $n<0$.
The conditions in the theorem are
designed to avoid logarithmic singularities
in the integral in \eqref{gamma}.  

\begin{remark}\label{Abel}
This is not a full characterization.  
Although we do not pursue this fully here,
essentially what is required is that 
$z^{2M} \frac{dw}{w}$ is an exact meromorphic
$1$-form on the Riemann surface associated
to the algebraic function $w=f(z)$.  $\spadesuit$
\end{remark}

The theorem gives a way to begin constructing
polynomials whose zero set has at least one branch
with the desired local behavior at $(0,0)$.
First, to get an algebraic function
with no zeros or poles in $\C_{*}$,
we need an irreducible 
polynomial $P(s,y) = \sum_{j=0}^{m}P_j(s)y^j \in \C[s,y]$ 
such that $P_m(s)$ has no zeros in $\C_{*}$ (to avoid
poles in $\C_{*}$)
and $P(s,0) = P_0(s)$ has no zeros in $\C_{*}$.
Necessarily, $P_0$ and $P_m$ are monomials.
Other than this, $P(s,y)$
needs to have an analytic branch through $(0,1)$
with logarithm satisfying 
the conditions on $L$ in Theorem \ref{mainthm}.

Section \ref{sec:global} goes through several examples
of $f(s)$ from Theorem \ref{algthm} such as
\begin{align*}
\text{Example \ref{revNext}:} \quad & ((1+4s^4)^{1/2}+2s^2)^{1/2} \text{ with polynomial } w-iw^2-iz^2\\
\text{Example \ref{ex:algsimple}:} \quad & 1+(1+is)^{1/2}\\
\text{Example \ref{ex:exotic}:} \quad & ((1-4s^4)^{1/2}+2is^2)^{1/2} \text{ with polynomial } w+w^2-iz^2
\end{align*}
We also go through several ``non-examples'' of Theorem \ref{algthm}
and Remark \ref{Abel} which essentially lead us through 
concrete instances of Abel's theory of integrals.

The table of contents gives an accurate road map for the paper.

\tableofcontents

\section{Simple zero on the boundary}

As discussed in the introduction
the unit ball $\B_d$ is biholomorphically equivalent to the Siegel half-space
\begin{equation}
U_d = \{(z_1,\dots,z_{d-1},w): \Im w > |z|^2\} 
\label{siegelupperhp}
\end{equation}
via the map 
\begin{equation}
F : U_d \to \B_d \qquad F(z_1,\dots, z_{d-1},w) 
= \left( \frac{2z_1}{i+w},\dots, \frac{2z_{d-1}}{i+w}, \frac{i-w}{i+w}\right)
\label{siegeltoball}
\end{equation}
\begin{equation}
F^{-1}: \B_d \to U_d \qquad F^{-1}(z_1,\dots,z_{d-1},w) 
= \left(\frac{iz_1}{1+w},\dots,\frac{iz_{d-1}}{1+w} , i\frac{1-w}{1+w}\right).
\label{balltosiegel}
\end{equation}

The domain $U_d$ and simple variations are referred to as 
\emph{unbounded realizations of the ball} in \cite{dA2}.
\footnote{The domain $U_d$ is one among many `classical' domains
studied by Siegel so perhaps a generic term is more appropriate.}
Thus, a polynomial $p(z_1,\dots, z_{d-1},w)$ is non-vanishing on $\B_d$ if and only if the polynomial  
\[
(i+w)^{\deg p} p( F(z,w))
\]
 is non-vanishing
on $U_d$.  
Also, a rational function $q/p$ is bounded near $(0,\dots,0,1)$ in $\B_d$ if and only if $q(F(z,w))/p(F(z,w))$
is bounded near $0$ within $U_d$.  

To begin we try to understand a simple boundary zero; 
namely where $p$ is non-vanishing
near $0$ in $U_d$, $p(0)=0$, and $\nabla p(0)\ne 0$.  
\begin{lemma} \label{psmooth}
Suppose $p(z,w) \in \C[z_1,\dots, z_{d-1},w]$ is
non-vanishing on a neighborhood of $0\in \C^d$ intersected with $U_d$.
Assume $p(0) = 0$ and $\nabla p(0) \ne 0$.
Then, $\frac{\partial p}{\partial z_j}(0) = 0$ for $j=1,\dots, d-1$
so that necessarily $\frac{\partial p}{\partial w}(0) \ne 0$.  
\end{lemma}

\begin{proof}
If say $\frac{\partial p}{\partial z_1}(0) \ne 0$, then
by the implicit function theorem
the zero set of $p$ can be parametrized locally 
via $z_1 = \phi(z_2,\dots, z_{d-1},w)$
for some analytic $\phi$ vanishing at $0$ and satisfying
\[
\Im w \leq |\phi|^2 + |z_2|^2 +\dots + |z_{d-1}|^2
\]
for all $(z_2,\dots,z_{d-1},w)$ sufficiently small
(since $p$ is non-vanishing in $U_d$).
Since the right hand side vanishes to order $2$,
we obtain $\Im w \leq 0$ for all sufficiently small $w$
which is impossible (pick $w \in i(0,\infty)$).
\end{proof}

Thus, for a simple zero the zero set of $p$ is
locally parametrized via $w = \phi(z_1,\dots, z_{d-1})$
satisfying
\[
\Im \phi(z_1,\dots, z_{d-1}) \leq |z|^2
\]
where $z = (z_1,\dots, z_{d-1})$.  
Just as above, $\phi$ can have no linear term.
If $\phi$ has a quadratic homogeneous term $\Phi_2(z)$,
it must satisfy 
\[
\Im \Phi_2(z) \leq 1
\]
for $|z|=1$.  

\begin{lemma} \label{Autonne}
Suppose we have a quadratic form $z^t A z$, here written using
a symmetric complex matrix $A$, 
satisfying
\[
\Im (z^t A z) \leq 1 
\]
for $|z|=1$.
Then, $A$ must be contractive.
\end{lemma}

\begin{proof}
By rotating $z$, the inequality is equivalent to the stronger inequality
\[
|z^t A z| \leq 1 \quad \text{ for } \quad |z|=1.
\]
By the Autonne-Takagi factorization (see \cite{HornJohnson} Corollary 4.4.4c), 
there exists
a unitary $U$ so that $U^t A U = D$ is diagonal with non-negative real entries 
$D_1,\dots, D_{d-1}$ on the diagonal.
Replacing $z$ with $Uz$ we have $|z^t D z| \leq 1$
or rather
\[
\left|\sum_{j=1}^{d-1} D_j z_j^2\right| \leq 1.
\]
We can rotate each component individually to see 
\[
\sum_{j=1}^{d-1} D_j |z_j|^2 \leq 1
\]
for $|z|=1$.  This is only possible if $D_j\leq 1$ for all $j$.
Thus, $D$ is contractive and so is $A$ since it has the same
operator norm.
\end{proof}

Thus, we can conclude the quadratic homogeneous term of $\phi$
is of the form $\Phi_2(z) = z^t A z$ where $A$ is a complex symmetric
contractive matrix.  
If $\|A\| <1$ then we get a converse statement;
namely, for $A$ symmetric with $\|A\| < 1$ and
\[
\phi(z) = z^tAz + \text{ higher order},
\]
the graph $w=\phi(z)$ avoids $U_d$ near $0$.  
Indeed, for such $\phi$ and for $(z,w) \in U_d$,
\[
|w- \phi(z)| \geq \Im w - \Im \phi(z) = 
\Im w - |z|^2 + |z|^2- \Im \Phi_2(z) - O(|z|^3)
\gtrsim |z|^2.
\]
Since $|\phi(z)| \lesssim |z|^2$, we also have $|w-\phi(z)| \gtrsim |w|$
within $U_d$.  Therefore,
\begin{equation} \label{dvarestimate}
|w-\phi(z)| \asymp |w|+ |z|^2 \asymp |w|
\end{equation}
for $(z,w) \in U_d$ near $0$.  Putting this together we have the following.

\smooththm*

As before, $(z)^2$ denotes the ideal generated
by $z^{\alpha}$ for $|\alpha|=2$ and $(w,(z)^2)$
denotes the ideal generated by $w$ and $(z)^2$.  

\begin{proof}
By the Weierstrass preparation theorem,
$p$ factors as $p = u(z,w)(w-\phi(z))$
where $u,\phi$ are analytic and $u(0,0) \ne 0$.
By Lemma \ref{psmooth}, $\nabla \phi(0) =0$ and $\frac{\partial p}{\partial w}(0)\ne 0$
and the quadratic homogeneous term $\Phi_2(z)$
of $\phi$ can be represented with a symmetric
matrix $A$ as $\Phi_2(z) = z^t A z$.
By computation, 
\[
A = -\frac{1}{2\frac{\partial p}{\partial w}(0)} Hp(0)
\]
and by assumption this matrix is strictly contractive.

By \eqref{dvarestimate}, every element $q$ of the ideal
$(w,(z)^2)$, has the property that $q/p$ is bounded 
near $0$ within $U_d$.  Conversely, if $q/p$ is bounded
near $0$ within $U_d$, then we wish to 
show $q \equiv 0$ modulo the ideal $(w,(z)^2)$.  
By the Weierstrass division theorem, we can write
\[
q = Q(z,w) w + R(z)
\]
where $Q(z,w), R(z)$ are analytic.  Reducing modulo
the ideal $(w,(z)^2)$, we can disregard $Q$ and assume
$R(z)$ has no terms of order 2 or higher.
Necessarily, if $R/p$ is locally bounded, $R(0)=0$.
So, $R(z) = \sum_{j=1}^{d-1} a_j z_j$ consists of a
linear term.
Since $p$ is locally comparable to $w$ near $0$,
we must have 
\[
\frac{\sum_{j=1}^{d-1} a_j z_j}{w}
\]
bounded in $U_d$ near $0$.  
Setting $\gamma(r) = (\overline{\sgn(a_1)}r,\dots, \overline{\sgn(a_{d-1})}r, i d r^2)$
which takes values in $\partial U_d$ for $r\in(0,\infty)$,
we have
\[
\frac{R(\gamma(r))}{idr^2} = \frac{\sum_{j} |a_j|}{idr}
\]
which goes to infinity as $r\to 0$ unless $R \equiv 0$.
This implies $R\equiv 0$ and $q\equiv 0$.
\end{proof}

The case where $A$ is contractive but not strictly contractive
is already difficult in two variables,
so we leave further study of $d>2$ variables
for the future.

\section{Two variables}
As discussed in the introduction, in two variables
the strategy is to analyze the function from \eqref{requirement}
\[
G_{\phi}(t) = G(t) = |t|^{2M} - \Im (\phi(t))
\]
near $0$.  
We will see that a good way to understand and characterize
the relevant $\phi$ is to study the critical
points of $\theta \mapsto G(re^{i\theta})$ 
for fixed $r$---the so-called angular critical points.
This also makes it possible to determine optimal lower bound behavior
\[
G(t) \gtrsim |t|^{K}
\]
when $G$ is non-negative around $0$.
If $G$ has an isolated zero at $0$ then such behavior is guaranteed 
by the Łojasiewicz inequality (cf. \cite{BKPS2}), but 
through our approach we can determine $K$ and
find curves along which $G$ is minimized.
It is also possible that $G$ is identically zero along
minimizing curves as would happen for the example $w-z^2$---
it is non-vanishing in $U_2$ but vanishes on 
the boundary at points $(\sqrt{i} r,ir^2)$.

How does this relate to finding admissible numerators?
With the given parametrization, $p$ will have Puiseux
factors
\[
w - \phi(\mu^{k} (\bar{c} z)^{1/M})
\]
for $\mu = \exp(2\pi i/M)$, and $k=0,\dots, M-1$.
Then, for $(z,w) \in U_2$
\begin{equation}\label{Puiseuxfactorestimate}
|w - \phi(\mu^{k} (\bar{c} z)^{1/M})| \ge \Im w - \Im \phi(\mu^{k} (\bar{c} z)^{1/M})
\ge
\Im (w) -|z|^2 + G(\mu^{k} (\bar{c} z)^{1/M}) \gtrsim |z|^{K/M}.
\end{equation}
This lower bound enables one to find some admissible numerators.
It requires additional work to show that an educated guess
for the entire ideal of admissible numerators is complete.

\subsection{Basic cases}\label{basiccase}

Some initial reductions are required to get into the most interesting
cases.  
Recall we have
\[
\phi(t) = \sum_{j=1}^{\infty} \phi_j t^j = \psi_0 t^{a_0} + o(t^{a_0})
\]
with $\psi_0 \in i(0,\infty)$.
First, we work through and find the two ``basic'' cases:
\begin{itemize}
\item $a_0 >2M$
\item $a_0=2M$ and $\psi_0 = i \tilde{\psi}_0$ with $0<\tilde{\psi}_0<1$
\end{itemize}
In Section \ref{sec-u}
we address the situation where $a_0=2M$ and $\psi_0 = i$.

\basictwovar*

\begin{proof}
If $a_0 < 2M$ then 
\[
\lim_{r\to 0} \frac{G(r e^{i\theta})}{r^{a_0}} = -\Im ( \psi_0 e^{i a_0\theta}) \geq 0
\]
which is false for $a_0 \theta = 3\pi/2$ since $\psi_0 \in i(0,\infty)$.

Now, if $a_0 > 2M$ then
\[
G(t) = |t|^{2M}( 1 - O(|t|^{a_0-2M})) \gtrsim |t|^{2M}
\]
is positive for $|t|$ sufficiently small.

When $a_0 = 2M$, we can write $t = r e^{i\theta}$ and
\[
G(t) = r^{2M}(1 - \Im(\psi_0 e^{i2M\theta})) - O(r^{2M+1}).
\]
If $G(t) \geq 0$ for $t$ small we have
\[
\lim_{r\to 0} \frac{G(r e^{i\theta})}{r^{2M}} = 1 - \Im(\psi_0 e^{i2M\theta}) \geq 0
\]
which implies $|\psi_0| \leq 1$.  

Conversely, `$a_0=2M$ and $|\psi_0|< 1$' leads to a sufficient condition and we have
\[
G(t) \gtrsim  |t|^{2M}.
\]

The upper bound $G(t) \lesssim |t|^{2M}$ follows in
both cases because $a_0 \geq 2M$.
By \eqref{Puiseuxfactorestimate}, $|w-\phi(\mu^k(\bar{c} z)^{1/M})| \gtrsim |z|^2$.
Since $\phi(\mu^{k} (\bar{c} z)^{1/M})$ vanishes to order $2$ or higher
we also see that
\[
|w - \phi(\mu^{k} (\bar{c} z)^{1/M})| \gtrsim |w|
\]
so that $|w - \phi(\mu^{k} (\bar{c} z)^{1/M})| \asymp |z|^2 + |w|$
since the upper bounds are trivial.
Another way to say this is that $w - \phi(\mu^{k} (\bar{c} z)^{1/M})$
can be replaced with a simple local model `$w$' since $|w| \asymp |z|^2 +|w|$
in $U_2$. 
\end{proof}

\subsection{Admissible numerators in basic case}\label{basiccaseideal}
The polynomial $p$ can be modeled
with a power of $w$ if all of $p$'s irreducible Weierstrass factors belong
to the basic case.

As discussed above, in the basic cases we have that the associated
Puiseux factors of $p$ satisfy
\[
|w - \phi(\mu^{k} (\bar{c} z)^{1/M})| \asymp |w| \asymp |w|+|z|^2
\]
near $(0,0)$ in $U_2$.
Thus, if $p$ has only a single irreducible Weierstrass factor of
the given type,
the admissible numerator ideal 
contains $(w,z^2)^M$ and we can
show that this ideal is everything.
Indeed, given $q \in \C[z,w]$ such that $q/p$ is bounded near $(0,0)$ 
within $U_2$ we claim that we can reduce $q$ mod $(w,z^2)^M$
and obtain zero.  
Since the Puiseux factors of $p$ are comparable to $w$, 
we need only assume $q/w^M$ is bounded near $(0,0)$.

By the Weierstrass division theorem we can 
assume that $q$ is a polynomial in $w$ with degree at most $M-1$.
Writing $q(z,w) = \sum_{j=0}^{M-1} q_j(z) w^j$ we can further
assume that $\deg q_j(z) < 2M-2j$ since $w^j z^{2(M-j)} \in (w,z^2)^M$.  

Note that if $(z,w) \in U_2$, then for any $r>0$, $(rz, r^2 w) \in U_2$.
If we consider the lowest power $k$ of $t$ appearing in $q(rz,r^2w)$,
it must be less than $2M$.  Writing 
\[
q(rz, r^2 w) = \tilde{q}(z,w) r^k + \text{ higher order in $r$}
\]
we have
\[
\left| \frac{q(rz,r^2w)}{r^{2M} w^M} \right| \gtrsim \frac{r^k |\tilde{q}(z,w)| - O(r^{k+1})}{r^{2M}}.
\]
If $q$ is not identically zero, we can choose a point where $\tilde{q}(z,w) \ne 0$
and sending $t\to 0$, we get a contradiction to local boundedness.
Therefore $q\equiv 0$ mod the ideal $(w,z^2)^M$.

\subsection{The non-basic cases: isolated and curve type}\label{sec-u}
Now let us go beyond the basic types.
For this we assume $a_0=2M$ and $\psi_0= i$. 
Then, we can write $\phi(t) = it^{2M} + o(t^{2M})$.
Our goal is to prove the following theorem which 
gives a complete characterization of the curves
above that avoid $U_2$ near $(0,0)$.

\mainthm*

The function $L(s)$ has the same structure 
as the parametrizing functions
that occur in the bi-upper half-plane setting
with some minor modifications.
Actually it exactly mirrors the setting of 
polynomials with no zeros in $\R\times \uhp$
near $(0,0)$ presented as Theorems 3.1 and 3.2 in \cite{Kprep}.
Again we have the feature that boundary
curves force smoothness for irreducible
Puiseux parametrizations but
in the present case there is no
notion of ``real stable'' polynomial 
(i.e. those with no zeros on $\uhp^2$ and $-\uhp^2$).

\begin{example} \label{excurve-u-new}
Setting $M=1$, $L(s) = is$, we obtain
the analytic curve parametrized by
\[
s \mapsto (s e^{is}, is^2 -(2/3)s^3)
\]
which avoids $U_2$ while it parametrizes 
a curve on $\partial U_2$
for $s$ real. $\diamondsuit$
\end{example}

\begin{example}\label{M2ex}
Set $M=2$ and set $L(s) = is^2+s^4+s^5$
which evidently satisfies the ``isolated type'' conditions.
The parametrization is
\[
s\mapsto (s^2 \exp(2is^2+2s^4+2s^5), is^4 -(4/3)s^6 + 2is^8+ (20/9)i s^9).
\]
Note the inclusion of $s^5$ in $L(s)$ prevents this example
from being reducible. $\quad \diamondsuit$
\end{example}

\begin{proof}[Proof of Theorem \ref{mainthm}]
To begin, for each fixed $r$ we analyze critical points
of $\theta \mapsto G(re^{i\theta})$, namely,
solutions to 
\[
\Im(\phi'(re^{i\theta}) ri e^{i\theta}) = 0.
\]
We call solutions
to $\Im(\phi'(t)it) = 0$ angular
critical points for $G$.  
We can solve for the angular critical points
in terms of power series operations on $\phi$.

Since $\phi(t) = it^{2M} + o(t^{2M})$
we have 
\[
\frac{1}{-2M} it\phi'(t) = t^{2M} + o(t^{2M})
\]
and we can factor this into 
\[
\frac{1}{-2M} it\phi'(t) = (\Phi(t))^{2M}
\]
where $\Phi(t)$ is analytic around $0$,
$\Phi(0)=0$, and $\Phi'(0)=1$.
Letting $\Psi(s) = \Phi^{-1}(s)$ be the analytic
inverse function we have
\begin{equation}
\phi'(\Psi(s))i \Psi(s) = -2M s^{2M}.
\label{phipsider}
\end{equation}
Thus, to solve $\Im(\phi'(t)it) = 0$ for $t$ near $0$
we can solve $0 = \Im( \phi'(\Psi(s))i \Psi(s)) = -2M \Im(s^{2M})$
which has the solution $s = r \nu^k$ for
$\nu = \exp(\pi i/2M)$, $k \in \mathbb{Z}$,
$r\in \R$.
In original coordinates the angular critical points
therefore consist of $t=\Psi(r \nu^k)$.

Note that by the chain rule and \eqref{phipsider}
\[
\phi(\Psi(s)) = \int_{[0,s]} 2Mi w^{2M} \frac{\Psi'(w)}{\Psi(w)} dw
\]
and writing $\Psi(s) = s \exp(L(s))$
for analytic $L(s)$ with $L(0) = 0$
we have
\begin{equation}\label{phipsi}
\phi(\Psi(s)) = is^{2M} + 2Mi\int_{[0,s]} w^{2M}L'(w)dw.
\end{equation}

To test if $G_{\phi}(t) \geq 0$ for small $t\in \C$
we only need to check at angular critical
points.  

Note first that
\begin{align*}
G_{\phi}(\Psi(s)) &= 
|\Psi(s)|^{2M} - \Im(\phi(\Psi(s)))\\
&=
|s|^{2M} \exp(2M\Re(L(s))) - 
\Im\left(is^{2M} + 2Mi \int_{[0,s]} w^{2M}L'(w)dw\right).
\end{align*}

Setting $s = r\nu^k$ for $r\in\R$ we calculate
$G$ at the angular critical points
\[
G_{\phi}(\Psi(r\nu^k)) 
=
r^{2M} \exp(2M\Re(L(r\nu^k)))
- (-1)^k r^{2M}
-2M\Re\left(\int_{[0,r\nu^k]} w^{2M} L'(w)dw\right).
\]
If $k$ is odd then the expression is $\gtrsim r^{2M}$
and non-negativity is guaranteed (i.e.
these critical points are likely local maxima)
so we only need to consider $k$ even.
Instead we can replace $\mu = \nu^2= \exp(i\pi/M)$
and consider $k\in \mathbb{Z}$.
So we now have
\begin{equation} \label{Gcritical}
G_{\phi}(\Psi(r\mu^k)) 
=
r^{2M} (\exp(2M\Re(L(r\mu^k)))
- 1)
-2M\Re\left(\int_{[0,r\mu^k]} w^{2M} L'(w)dw\right).
\end{equation}

For a fixed $k$, this expression is
non-negative for all sufficiently
small $r\in \R$ if and only if 
either the expression is identically zero
or there is a lowest order term of 
the form $\beta r^{2K}$ for $\beta>0$.
Let us label the coefficients $L(s) = \sum_{j=1}^{\infty} L_j s^j$ like so.
If $\Re(L_j \mu^{jk}) = 0$ for all $j$
then necessarily \eqref{Gcritical} is 
identically zero.
Let us define for each $k$
\[
m(k) = \min\{j\in\mathbb{N}: \Re(L_j \mu^{kj}) \ne 0\}
\]
and we allow for the possibility that $m(k) = \infty$.
Using this first nonvanishing term we can calculate the
lowest order term in \eqref{Gcritical}.
Here are the ingredients; 
we have $\Re(L(r\mu^k)) = \Re(L_{m(k)}\mu^{k m(k)})r^{m(k)} + o(r^{m(k)})$
so that
\[
\exp(2M\Re(L(r\mu^k))) - 1 = 2M \Re(L_{m(k)}\mu^{km(k)}) r^{m(k)} + o(r^{m(k)})
\]
as well as
\[
\Re\left(\int_{[0,r\mu^{k}]} w^{2M} L'(w) dw \right)= \frac{m(k)}{m(k)+2M} \Re(L_{m(k)} \mu^{k(m(k)+2M)}) r^{m(k)+2M} + o(r^{m(k)+2M}).
\]

Putting this all together 
\begin{equation}\label{Gnonneg}
G_{\phi}(\Psi(r\mu^k)) 
=
\frac{(2M)^2}{m(k)+2M} r^{2M+m(k)} \Re(L_{m(k)} \mu^{km(k)}) + o(r^{2M+m(k)}).
\end{equation}
This expression is non-negative for small $r$ 
if and only if $m(k)$ is even
and
\begin{equation}\label{Lmk}
\Re(L_{m(k)} \mu^{km(k)}) >0.
\end{equation}

Thus, we conclude that $G_\phi(t)$ is non-negative
for small $t\in \C$ if and only if 
for every $k\in \mathbb{Z}$, either $m(k) = \infty$
or $m(k)<\infty$, $m(k)$ is even, and \eqref{Lmk} holds.
We now wish to refine this characterization
and show that $G_\phi(t)$ is non-negative for small $t\in \C$
if and only if exactly one of the following cases holds
\begin{description}
    \item[Case A] $M=1$ and $m(0)=\infty$ OR
    \item[Case B] for all $k$, $m(k)<\infty$, $m(k)$ is even,
    and \eqref{Lmk} holds.
\end{description}
To do this we show that if $m(\ell)=\infty$ for some $\ell$ 
then necessarily $M=1$.

\textbf{Claim}: If $m(\ell) = \infty$ for some $\ell$ then $M=1$.

To see this, we can suppose $\ell=0$ --- namely $m(0)=\infty$ ---
and hence $L(s)$ has all coefficients belonging 
to $i\R$.
(Else, consider $L(s\mu^{\ell})$.)
Assuming $M\ne 1$, we will now show $m(1)<\infty$.
Indeed, if $L(s\mu)$ 
has all coefficients belonging to
$i\R$ then we have $\Re(L_j)= \Re(L_j\mu^j)=0$
for all $j$.  Necessarily $L_j=0$ except
when $\mu^j=\pm 1$ which occurs
exactly when $j$ is a multiple of $M$.
So, $L_j=0$
unless $j$ is a multiple of $M$.  
In this case $L(s) = \tilde{L}(s^M)$ 
for some analytic $\tilde{L}$.
We can trace this back to $\phi$ 
and contradict injectivity of our Puiseux parametrization.
We have $\Psi(s) = s e^{\tilde{L}(s^M)}$
which implies the inverse function
has a similar form, namely $\Phi(t) = t \tilde{\Phi}(t^M)$ for some analytic $\tilde{\Phi}$.
This follows from the Lagrange inversion formula. 
Then, $it \phi'(t) = -2M t^{2M} \tilde{\Phi}(t^{M})^{2M}$ is an analytic function of $t^M$ and hence so is $\phi(t)$
and this contradicts injectivity of our Puiseux parametrization.  
Thus, $m(1)<\infty$ since  $M\ne 1$.

Now we can consider
\begin{equation}\label{considerm}
m = \min\{m(k): k\in \mathbb{Z}\}
\end{equation}
which is necessarily finite since $m(1)<\infty$ (though $m(0)=\infty$).
We must have that $m$ is even
and
\begin{equation}\label{Relm}
\Re(L_m \mu^{km}) \geq 0
\end{equation}
for all $k$ by \eqref{Gnonneg} and \eqref{Lmk}
---in addition
we have strict inequality for some $k$.
Indeed, if $m=m(k)$ then we have strict inequality
and if $m<m(k)$ then $\Re(L_m\mu^{mk})=0$.

The following elementary lemma,
whose proof we omit, isolates a situation occurring several times in what follows.

\begin{lemma}\label{sigmalemma}
    If $\sigma$ is a root of unity and $A\in\C\setminus\{0\}$,
    suppose $\Re(A\sigma^k) \geq 0$ for all $k\in\mathbb{Z}$.
    Then, either $\sigma = -1$ and $\Re(A) = 0$
    or $\sigma = 1$ and $\Re(A)\geq 0$.      
\end{lemma}

Assuming $m(0)=\infty$ we have $\Re(L_m) = 0$
so the lemma would say $\mu^m = -1$,
but this clearly contradicts the fact that $\Re(L_m \mu^{mk}) > 0$
for some $k$.
This contradiction proves $M=1$.  (End of proof of claim.) $\clubsuit$

Now case A simply means $M=1$ and $L(s) \in i\R[s]$.

We can say more about case B.
We can again define $m$ as in \eqref{considerm},
and we again have \eqref{Relm} for all $k$ with
strictly inequality for some $k$.  
But now referring to Lemma \ref{sigmalemma}
we must have $\mu^m=1$ and $\Re(L_m)>0$
(the other case, $\mu^m=-1$ and $\Re(L_m)=0$ implies $\Re(L_m \mu^{mk}) = 0$
for all $k$).
This means $m$ is a multiple of $2M$.
Thus, in the case at hand of $m(k)<\infty$ for all $k$,
we have $\Re(L_m \mu^{mk}) = \Re(L_m) >0$
so that $m=m(k)$ for all $k$.  
So, this case is equivalent to the simpler condition 
that $m$ (as in \eqref{considerm}) 
is a multiple of $2M$ and $\Re(L_m)>0$.

Still in case B, we have that for $j<m$,
$\Re(L_j \mu^{jk}) = 0$ for all $k$
which implies either $L_j=0$ or $j$ is a multiple of $M$
and $\Re(L_j) = 0$.
This is equivalent to $L(s)$ having the form
\[
L(s) = iA_0(s^M) + s^{2MK}L_1(s)
\]
where $K\in \mathbb{N}$, 
$A_0(s) \in \R[s]$ has degree less than $K$,
and $L_1 \in \C\{s\}$, $\Re(L_1(0))> 0$.

To reiterate, $G_\phi(t)$ is non-negative
for all small $t\in \C$ if and only if
we have one of the two cases
\begin{description}
    \item[Case A] $M=1$ and $L(s) \in i\R[s]$ OR
    \item[Case B] $L(s) = iA_0(s^M) + s^{2MK}L_1(s)$ with all the details
    just stated.
\end{description}
This concludes the proof.
\end{proof}

\begin{remark} \label{Remark:buildL}
To go between the standard Puiseux type parametrization
\[
t\mapsto (ct^M, \phi(t))
\]
with $\phi(t) = it^{2M}+o(t^{2M})$
and our new parametrization
we simply factored 
\[
it \phi'(t)/(-2M) = (\Phi(t))^{2M}
\]
with $\Phi(0)=0,\Phi'(0)=1$.
Set $\Psi(s) = \Phi^{-1}(s)$ and then let $L(s) = \log(\Psi(s)/s)$
where $L(0)=0$. $\spadesuit$
\end{remark}

If we want to check if a given standard
Puiseux parametrization $t\mapsto (ct^M, \phi(t))$
avoids $U_2$ locally, it helps to have some simple criteria to
verify this.

\begin{corollary}\label{simplecriterion}
    If $\phi(t) = it^{2M} + \alpha t^{2M(K+1)} + o(t^{2M(K+1)})$
    where $\Im \alpha <0$, then $t\mapsto (ct^M, \phi(t))$
    avoids $U_2$ for $t$ small
    and the function $L(s)$ associated to $\phi$ begins 
    \[
    L(s) = i\frac{K+1}{2M} \alpha t^{2MK} + o(t^{2MK}).
    \]
\end{corollary}

\begin{proof}
We simply need to calculate the first term in $L(s)$
from the construction above.
We have 
\[
it \phi'(t)/(-2M) = t^{2M} -i\alpha (K+1)t^{2M(K+1)} + o(t^{2M(K+1)})
=(\Phi(t))^{2M}
\]
where
\[
\Phi(t) = t\left(1- i\frac{K+1}{2M} \alpha t^{2MK}+ o(t^{2MK})\right)
\]
and the inverse function
\[
\Psi(s) = s\left(1+ i\frac{K+1}{2M}\alpha s^{2MK} + o(s^{2MK})\right) = s e^{L(s)}
\]
where
\[
L(s) = i\frac{K+1}{2M} \alpha s^{2MK} + o(s^{2MK}).
\]
Since $\Re(i\alpha) = -\Im \alpha>0$, Theorem \ref{mainthm}
implies that $t\mapsto (ct^M, \phi(t))$ avoids $U_2$ for $t$ small.
\end{proof}

Other criteria are possible; it is mostly a matter of how
deep into the power series computations one is willing to go.

\subsection{Admissible numerators for isolated type}\label{admnum}

The first step in investigating admissible numerators
is to obtain the lower bound behavior of $G_{\phi}(t)$.

\begin{theorem} \label{thm:Glower}
    For $\phi$ as in Theorem \ref{mainthm} in the isolated
    type case, we have
    \[
    G_{\phi}(t) \gtrsim |t|^{2M(1+K)}.
    \]  
\end{theorem}

\begin{proof}
For $t=r e^{i\theta}$, $G_{\phi}(t)$ is
minimized with respect to $\theta$ at an angular
critical point with modulus $r$.  
The angular critical points are of the form
$\Psi(x \mu^k)$ for $x\in\R$,
so we need to see the rough dependence of $x$
on $r$ for each $k$.
So we solve $|\Psi(x\mu^k)| = x e^{\Re(L(x\mu^k))} = r$
for $x$ in terms of $r$.
Since $\Re(L(x\mu^k)) = \Re(L_{2MK}) x^{2MK} + o(x^{2MK})$,
we see that 
\[
x_k(r) = r\left(1 - \Re(L_{2MK}) r^{2MK} + o(r^{2MK})\right)
\]
for each $k$ (and the dependence on $k$ is hidden in $o$
term).
Thus,
\[
G_{\phi}(r e^{i\theta})
\geq \min_{k} G_{\phi}(\Psi(x_k(r)\mu^k))
\gtrsim
\min_k x_k(r)^{2M+2MK} \asymp r^{2M+2MK}
\]
by \eqref{Gnonneg}.
\end{proof}

\begin{corollary} \label{cor:compare}
    For $\phi$ as in Theorem \ref{mainthm} in the isolated type case,
    $\phi$ has the form
    \[
    \phi(t) = \phi_0(t^M) + O(t^{2M(1+K)})
    \]
    where $\phi_0$ is a polynomial of degree less than $K+1$.
    In addition,
    \[
    |w - \phi((\bar{c}z)^{1/M})| \asymp |w - \phi_0(\bar{c} z)| + |z|^{2(1+K)}.
    \]
    for $(z,w) \in U^2$ near $(0,0)$.
\end{corollary}

\begin{proof}
Recall that Remark \ref{Remark:buildL} shows
how to go between $L(s)$ and $\phi(t)$ via
associated functions $\Phi, \Psi$.
We can calculate the coefficients
of $\Phi = \Psi^{-1}$ using the Lagrange inversion formula.

Using the notation $[t^n]\Phi(t)$ to denote extracting
the coefficient of $t^n$, we have
\[
[t^n]\Phi(t) = \frac{1}{n} [s^{n-1}] e^{-nL(s)}
\]
using the formulation of Lagrange inversion found
in \cite[p.66]{AnalyticCombinatorics}.
For $j\leq 2MK$, the term $s^{j}$ appears in $e^{-nL(s)}$
only when $j$ is a multiple of $M$
since $L(s)$ has the form given in Theorem \ref{mainthm}.
Thus, for $n\leq 2MK+1$, the only powers $t^n$ 
appearing in $\Phi(t)$ must be those with $n-1$ a multiple
of $M$.
Therefore, $\Phi$ takes the form
\[
\Phi(t) = t(B_0(t^M) + O(t^{2MK}))
\]
where $B_0$ is a polynomial of degree less than $2K$.
In turn, since
\[
it\phi'(t)/(-2M) = (\Phi(t))^{2M} = t^{2M}(B_0(t^M) + O(t^{2MK}))^{2M}
\]
we see that $\phi(t)$ contains only powers $t^j$ with
$j$ a multiple of $M$ for $j\leq 2M+2MK$.
Therefore, we may write
\[
\phi(t) = \phi_0(t^M) + O(t^{2M(1+K)})
\]
where $\phi_0(t) \in \C[t]$ has degree less than $1+K$.
Then, $\phi((\bar{c}z)^{1/M}) = \phi_0(\bar{c}z) + O(z^{2(1+K)})$
although the big-O term could involve fractional powers.

By \eqref{Puiseuxfactorestimate} and Theorem \ref{thm:Glower},
we have
\[
|w-\phi((\bar{c}z)^{1/M})| \gtrsim |z|^{2(1+K)}.
\]
for $(z,w) \in U_2$ near $(0,0)$.  
From this we deduce
\[
|w - \phi((\bar{c}z)^{1/M})| \asymp |w - \phi_0(\bar{c} z)| + |z|^{2(1+K)}
\]
for $(z,w)$ in $U_2$ and close to $(0,0)$.
\end{proof}

We can now prove the following.

\admissiblenonbasic*

\begin{proof}
Let $\mathcal{I} =(w-\phi_0(\bar{c} z), z^{2(1+K)})^{M}$
and let $\mathcal{I}_{p}^{\infty}$ denote the
ideal in $\C\{z,w\}$ of $q\in \C\{z,w\}$
such that $q/p$ is bounded near $(0,0)$ within
$U_2$.

Our assumption on $p$ means that
it factors as
\[
p(z,w) = u(z,w) \prod_{j=1}^{M}(w- \phi(\nu^j(\bar{c}z)^{1/M})
\]
with $u(z,w) \in  \C\{z,w\}$, $u(0,0)\ne 0$,
$\nu =\exp(2\pi i/M)$, and $\phi,c, M$ as in Theorem \ref{mainthm}.
The constant $c$ plays no role in the proof so we assume
$c=1$.

By Corollary \ref{cor:compare}, 
we have for $(z,w) \in U_2$ near $(0,0)$
\[
|w-\phi(z^{1/M})| \asymp |w-\phi_0(z)| + |z|^{2(1+K)}
\]
for any branch of $z^{1/M}$ which in particular implies
that any two branches are comparable in $U_2$:
\[
|w-\phi(z^{1/M})| \asymp |w- \phi(\nu^j z^{1/M})|.
\]
This implies 
\[
|p(z,w)| \asymp (|w-\phi_0(z)| + |z|^{2(1+K)})^M
\]
as well as
\[
(w-\phi_0(z))^{j} z^{(M-j)2(1+K)}\in \mathcal{I}_{p}^{\infty}
\]
for $j=0,\dots, M$.
In other words, $\mathcal{I} \subset \mathcal{I}_{p}^{\infty}$.

To prove the reverse, we can take $q(z,w) \in \mathcal{I}_{p}^{\infty}$ and show that $q$ modulo $\mathcal{I}$ equals
$0$.
Writing $q(z,w+\phi_0(z))$ as a power series
at $(0,0)$, we can then write
\[
q(z,w) = \sum_{j=0}^{\infty} q_j(z) (w-\phi_0(z))^{j}
\]
for $q_j(z) \in \C\{z\}$.  Reducing mod $\mathcal{I}$
we can truncate the series at $M-1$ and
assume each $q_j$ is a polynomial of degree less than
$2(1+K)(M-j)$:
\[
q(z,w) = \sum_{j=0}^{M-1} q_j(z) (w-\phi_0(z))^{j}.
\]

We now build a family of (real parameter) curves in $U_2$
that can be used to show each $q_j = 0$.
Define
\[
z(r) = r^M e^{MiA_0(r^M)} 
\]
\[
w(r,B) = ir^{2M} - 2M\int_{0}^{r}x^{2M} \frac{d}{dx}(A_0(x^M)) dx + i Br^{2M(K+1)}
\]
where $r\in\R$, $B>0$ is a parameter, and $A_0$
is the real-coefficient
polynomial of degree less than $K$ 
from Theorem \ref{mainthm}.
Notice that 
\[
r \mapsto (z(r), w(r,B))
\]
defines a real 1-dimensional curve in $U_2$
since $\Im w(r) = r^{2M}+ Br^{2M(K+1)} > |z(r)|^2 = r^{2M}$.
These curves are essentially a perturbation of zeros
of $p$.
Now,
\[
\phi(r e^{iA_0(r^M)}) = \phi_0(z(r)) + O(r^{2M(1+K)})
\]
but also since $L(r) = iA_0(r^M) + r^{2MK}L_1(r)$
we have 
$e^{L(r)} = e^{iA_0(r^M)} + O(r^{2MK})$ and therefore
\begin{align*}
\phi(r e^{L(r)}) 
& = \phi_0(r^M e^{ML(r)}) + O(r^{2M(1+K)}) \\
&= \phi_0(r^M e^{iMA_0(r^M)} + O(r^{2MK+M})) + O(r^{2M(1+K)}) \\
&=
\phi_0(r^M e^{iM A_0(r^M)}) + O(r^{2M(1+K)}) = \phi_0(z(r)) +
O(r^{2M(1+K)})
\end{align*}
with the last line coming from the fact that $\phi_0(t)$ begins with $it^2$.  
On the other hand, by \eqref{phipsi}
\[
\phi(r e^{L(r)}) = ir^{2M} + 2Mi\int_{0}^{r} x^{2M} L'(x)dx
=
w(r,0) + O(r^{2M(1+K)}).
\]
Therefore,
\[
\phi_0(z(r)) = w(r,0) + r^{2M(1+K)}F(r)
\]
for some power series $F(r)$.

Now, since $|p(z,w)| \asymp (|w-\phi_0(z)|+|z|^{2(1+K)})^M$,
evaluating along $(z(r),w(r,B))$ we have
\[
|p(z(r),w(r,B))|
\lesssim (|iB-F(r)|r^{2M(1+K)} + r^{2M(1+K)})^M
\lesssim r^{2M^2(1+K)} (B+1)^M
\]
for $B>0$. 
Assuming $|q/p| \leq C$ within $U_2$ near $(0,0)$
we now have
\[
|q(z(r),w(r,B))| \lesssim r^{2M^2(1+K)} (B+1)^M.
\]
But
\[
q(z(r),w(r,B)) = \sum_{j=0}^{M-1}q_j(z(r))(iB -F(r))^j r^{2Mj(1+K)}
\]
Now we consider the lowest order terms that
can appear.  
Writing $q_j(z) = c_j z^{a_j} + o(z^{a_j})$,
the lowest order term of $q_j(z(r))$ is $c_j r^{Ma_j}$
and hence $q_j(z(r)) r^{2Mj(1+K)}$ has lowest order
term $c_j r^{Ma_j+2Mj(1+K)}$.
Consider now
\[
a= \min_j \{Ma_j+2Mj(1+K)\} < 
2M^2(1+K)
\]
and let $J$ be the set of minimizing indices.
Assuming not all $q_j$ are zero, $J$ is nonempty.
Then, the lowest order term of $q(z(r),w(r,B))$
is
\[
r^a\sum_{j\in J} c_j(iB)^j.
\]
But since $|q(z(r),w(r,B))|$ vanishes to order $2M^2(1+K)$
for each $B$, this term must equal zero for each $B>0$.
Therefore, each $c_j$ must be zero contrary to
their construction.  
Thus, we must have that $J$ is empty and every $q_j$ is zero
which means that $q = 0$ mod $\mathcal{I}$.
\end{proof}

\section{Global examples}\label{examplesection}
In this section, we examine specific examples of globally stable polynomials that illustrate our findings. A similar analysis is carried out in the bidisk in \cite{BKPS} and in polydisks in \cite{BKPS2}.

While several methods have been devised to construct and modify multivariate polynomials that are stable in a polydisk, see for instance \cite{pascoe18,BPS2,BH,TDetal}, this aspect of the theory seems less developed in the settings of the unit ball or the Siegel half-plane.

\subsection{Elementary examples}
\begin{example}
We revisit the simplest polynomial that is stable in the $2$-ball, but has a boundary zero, namely $p(z_1,z_2)=1-z_2$. 

Its counterpart in the Siegel half-plane is 
\[P(z,w)=w.\]
The set $Z_p \cap U_2$ is locally (and in this case globally) parametrized by $t\mapsto (t,\phi(t))$ with $\phi\equiv 0$, so $\phi$ is trivial. As we worked out in the Introduction, for this example
\[\mathcal{I}_p^{\infty}=(w, z^2) \quad \diamondsuit\]
\end{example} 

\begin{example} \label{Nextsimplest}
The next simplest example for the ball is
\[p_c(z_1,z_2)=1-z_2+cz_1^2,\]
where $c\geq 0$ is a constant. That these polynomials are globally stable with respect to $\mathbb{B}_2$ precisely when $0\leq c\leq \frac{1}{2}$ can be seen by noting that, on $Z_{p_c}$, we have $z_2=1+cz_1^2$, and with $z_1=x_1+iy_1$, 
\[|z_1|^2+|z_2|^2=|z_1|^2+|1+cz_1^2|^2=(1+2c)x_1^2+(1-2c)y_1^2+2c^2x_1^2y_1^2+c^2(x_1^4+y_1^4)+1.\]
When $0\leq c\leq \frac{1}{2}$, all terms on the right are positive for $(x_1,y_1)\neq (0,0)$ and the sum attains a minimum of $1$ at $(0,0)$. Conversely, if $c>\frac{1}{2}$ the expression can be made smaller than $1$ by picking $x_1=0$ and $y_1$ sufficiently close to the origin.
Hence we have $Z_{p_c}\cap \left(\overline{\mathbb{B}_2}\setminus\{(0,1)\}\right)=\emptyset$ precisely when $0\leq c\leq \frac{1}{2}$.
These polynomials have previously featured in \cite[Section 2.8]{KZ}
in the context of 3 point interpolation for bounded analytic functions
on the ball.

Moving to $U_2$ using the map $F$ in \eqref{siegeltoball}, and rescaling, we obtain the polynomial
\begin{equation}
P_c(z,w)=w-iw^2-2ciz^2,  
\label{basicexfamily}
\end{equation}
which is stable in $U_2$ for $0<c\leq \frac{1}{2}$ by the computation above. Setting $P_c(z,w)=0$, completing the square, and identifying the correct branch, we find that $Z_{P_c}$ is locally parameterized near $(0,0)$ by $t\mapsto (t,\phi_{c}(t))$, where
\[\phi_c(t)=\frac{i}{2}\left((1+8ct^2)^{\frac{1}{2}}-1\right)=2cit^2-4c^2it^4+16c^3it^6+\cdots\]
In other words, we have $\psi_0 = 2ci$.
When $0<c<\frac{1}{2}$, we are in the basic case $M=1$, $a_0=2M=2$, and $0<\Im \psi_0<1$, meaning that
\[\mathcal{I}_{P_c}^{\infty}=(w,z^2).\]

When $c=\frac{1}{2}$, we get $\psi_0=i$, and $\phi_{1/2}(t) = it^2 -it^4+ o(t^4)$.
By Corollary \ref{simplecriterion}, $L(s) = s^2 + o(s^2)$.
Thus, $K=1$ and to calculate the admissible numerator
ideal we record the Taylor polynomial of $\phi_{1/2}$ of degree less than $4$, namely $iz^2$.
By Theorem \ref{admissiblenonbasic}, the ideal of admissible numerators is
\[\mathcal{I}_{P_\frac{1}{2}}^{\infty}=(w-iz^2, z^4). \qquad \diamondsuit\]
\end{example}

\begin{example} \label{Deg4}
A similar analysis allows us to determine parameter values for which the polynomials 
\[q_c(z_1,z_2)=1-z_2+cz_1^4, \quad c\geq 0,\]
are stable with respect to $\mathbb{B}_2$. On $Z_{q_c}$, we have $z_2=1+cz_1^4$ and a computation reveals that
\[|z_1|^2+|z_2|^2=1+x_1^2+y_1^2-12cx_1^2y_1^2+c^2(x_1^4+y_1^4-6x_1^2y_1^2)^2+2c(x_1^4+y_1^4).\]
A sufficient condition for having $|z_1|^2+|z_2|^2\geq 1$, 
with equality if and only if $(z_1,z_2)=(0,1)$ is $0\leq c\leq \frac{1}{12}$, 
as can be seen by examining when $(x_1,y_2)\mapsto x_1^2+y_1^2-12cx_1^2y_1^2$ 
is non-negative for $x_1^2+y_1^2\leq 1$. 
(With additional effort, one can show that $0\leq c\leq \frac{27}{32}$
is necessary and sufficient for stability, 
but in this full range additional zeros on $\partial\mathbb{B}_2$ appear.)

Moving to $U_2$, we arrive at the family
\[Q_c(z,w)=w-3iw^2-3w^3+iw^4+8ciz^4,\]
which furnishes stable polynomials for $U_2$, with a single zero at $(0,0)$, when $0\leq c\leq \frac{1}{12}$.

Note that $w-3iw^2-3w^3+iw^4$ has a simple root at $w=0$ (and a triple root at $w=-i$). Then $Z_{Q_c}$ is locally parameterized near $(0,0)$ by $t\mapsto (e^{i\pi/4}t,\phi_c(t))$, with
\[
\phi_c(t)=8cit^4+\cdots, 
\]
and we are again in the basic case $M=1$ but now in the instance where $a_0=4>2M$.
By Section \ref{basiccaseideal}, the ideal of admissible numerators is
\[I^{\infty}_{Q_c}=(w,z^2). \quad \diamondsuit\]
\end{example}

\subsection{Stable polynomials via one-variable polynomials}\label{1dlift}
Any polynomial $\mathfrak{p}\in \C[x]$ that is stable with respect to the unit disk $\D=\{x\in \C\colon |x|<1\}$ lifts to a polynomial that is stable in $\B_d$ by setting $p(z_1,\ldots, z_d)=\mathfrak{p}(z_d)$, and zeros $\{x_0,x_1,\ldots,x_n\}$ on $\mathbb{T}=\{x\in\C \colon |x|=1\}$ are transplanted to points 
$\{(0, \ldots,0, x_j)\}\subset \partial\mathbb{B}_d$. 

A slightly more interesting lifting (which has featured in several works, eg. \cite[Section 3]{Arv} and \cite[Section 4]{APRS}) arises from the map $\tau \colon \B_d\to \C$ given by
\[\tau(z_1,\ldots, z_d)=d^{d/2}\prod_{j=1}^dz_j.\]
Note that $|\tau(z)|\leq |z|^{d}$ since
\[(|z_1|^2\cdots |z_d|^2)^{\frac{1}{d}}\leq \frac{1}{d}|z|^2\]
by the arithmetic-geometric mean inequality, with equality precisely when all moduli are equal. 
Now if $\mathfrak{p}\in \mathbb{C}[x]$ is stable with respect to $\D$, 
then $p(z)=(\mathfrak{p}\circ \tau)(z) \in\mathbb{C}[z_1,\ldots, z_d]$ is stable in $\B_d$, 
and if $\mathfrak{p}(x_0)=0$ for some $x_0\in \mathbb{T}$ 
then $p$ vanishes along a curve $\tau^{-1}(\{x_0\})\subset \partial \mathbb{B}_d$ 
where all variables have equal modulus.
A prime example of this is the two-variable polynomial 
\[p(z_1,z_2)=1-2z_1z_2,\]
which arises from $\mathfrak{p}(x)=1-x$, and variants obtained by composition with unitaries (eg. \cite[pp. 251-253]{S}) such as $p=1-z_1^2-z_2^2$
(which also appears in the next section). 
A particularly attractive version of these polynomials in the unbounded realization of the ball is
\[P(z)=w-z^2.\]

We cannot get any genuinely different two-variable (or multivariable) examples using this one-variable lifting approach since any one-variable polynomial can be factored into linear factors of the form $a-z$, possibly with repetition, and then $p\circ \tau$ and its unbounded analogs will also factor into expressions of the above form.

\subsection{Examples from quadratic forms} \label{quad}

If $A$ is a symmetric contractive $d\times d$ matrix, then $P(z) = 1- z^tAz$
has no zeros in $\B_d$.  As in the proof of Lemma \ref{Autonne},
$A$ can be factored as $A = U D U^t$ where $D$ is a diagonal
matrix with entries $D_1,\dots, D_d \in [0,1]$ on the diagonal 
and $U$ is a unitary.
Then, up to composition with a unitary, $P$ is of the form 
\[
1- z^tDz = 1- \sum_{j=1}^{d} D_j z_j^2.
\]
The only way to have a boundary zero is if some $D_j =1$.  
If several $D_j=1$ then we get an entire
real subspace intersected with $\partial \B_d$
as boundary zeros.  
Assuming $D_d = 1$ and converting to $U_d$ we obtain
the polynomial 
\[
p(z_1,\dots, z_{d-1},w) = iw - \sum_{j=1}^{d-1} D_j z_j^2.
\]
The expression from Theorem \ref{smooththm}, $\frac{1}{2|\nabla p(0)|} Hp(0)$
is a diagonal matrix with $D_1,\dots, D_{d-1}$ on the diagonal.

The most extremal case $P(z) = 1 - z^tz$
is of some interest as it has boundary zero set $(\partial \B_d) \cap \R^d$.
It is worth pointing out that automorphisms of the ball
only produce unitary transformations of this polynomial.
Indeed, every automorphism of $\B_d$ can be written
as $U \circ \varphi_a$ where $U$ is a unitary transformation
and $\varphi_a$ is the involutive automorphism
\[
\varphi_a(z) = \frac{a - P_a z - s_a Q_a z}{1-\langle z, a\rangle}
\]
where we write as in Rudin \cite{Rud}; $a\in \B_d$, $P_a$ is
orthogonal projection onto the span of $a$, $Q_a = I-P_a$, and
$s_a = (1-|a|^2)^{1/2}$.
The only thing relevant here is the formula
\begin{equation} \label{Rudiniii}
1- \langle \varphi_a(z), \varphi_a(w)\rangle=
\frac{(1-|a|^2)(1-\langle z,w\rangle}{(1-\langle z,a\rangle)(1-\langle a,w\rangle)}
\end{equation}
which is Theorem 2.2.2(iii) of \cite{Rud} (page 26).

We can equally well write any automorphism of $\B_d$
as $\varphi_a \circ U$ for some unitary, so 
considering the (now rational) function
\[
P\circ \varphi_a \circ U
\]
up to unitary composition, we can ignore the unitary $U$.
We can compose with a diagonal unitary to
force $a \in \R^d \cap \B_d$ and in doing so $\varphi_a(z)$
is real for $z\in\R^d\cap \B_d$.  
So, for real $z=w$ equation \eqref{Rudiniii} becomes
\[
1- \varphi_a(z)^t \varphi_a(z)=
\frac{(1-|a|^2)(1-z^tz)}{(1-z^ta)^2}.
\]
But since both sides are analytic in $z$ and the identity holds for real
$z$, it must hold for all $z$.
This then shows $P\circ \varphi_a$ is $1-z^tz$ after clearing
denominator and constant factor.

\subsection{Rudin's construction} \label{Rudin}

In Rudin \cite[p.164]{Rud}, a simple construction is
given for producing many polynomials with no zeros on $\B_2$
but with a zero at $(0,1)$.
Given that the power series
\[
1- \sqrt{1-t} = \sum_{k=1}^{\infty} c_k t^k
\]
has $c_k>0$ for all $k\geq 1$, 
if we consider any analytic function of the form
\[
f(z_1,z_2) = z_2 + z_1^2 g_1(z_1,z_2) + z_1^4 g_2(z_1,z_2)+\cdots
\]
where each $g_k$ is analytic on $\B^2$ and satisfies $\sup_{\B_2}|g_k| \leq c_k$
then for $(z_1,z_2) \in \B_2$
\[
|f(z_1,z_2)| \leq |z_2| + 1-\sqrt{1-|z_1|^2} < 1.
\]
Then, $F(z_1,z_2) := 1-f(z_1,z_2)$ has no zeros in $\B_2$
and if $f$ extends analytically to $(0,1)$
we have $f(0,1) = 1$, so that $F(0,1) = 0$.
Noting $c_k = \frac{1}{2k} \frac{1}{4^{k-1}} \binom{2k-2}{k-1}$
and in particular
\[
c_1 = 1/2,\quad c_2 = 1/8, \quad  c_3 = 1/16, \quad c_4 = 5/128,\dots
\]
we can choose the $g_k(z_1,z_2)$ to be polynomials
(and zero for all but finitely many $k$) to get examples.
Here are some examples obtained by choosing the $g_k$ to be constant.
\[
1- \left(z_2 + \frac{1}{2}z_1^2\right), \quad 1- \left(z_2 + \frac{1}{2}z_1^2 + \frac{1}{8} z_1^4\right),\dots.
\]
It is possible to show that when we convert to $U_2$, 
we obtain curves with parametrizations
\[
t\mapsto (it,\phi(t)) = (it, it^2 - i2\beta t^{2K+2} + o(t^{2K+2}))
\]
for $\beta >0$ (the $2\beta$ is for convenience later).  The example $1-z_2-(1/2)z_2^2$ corresponds to $K=1$
and higher order truncations lead to larger $K$.

By Corollary \ref{simplecriterion}, $L(s) = \beta (K+1) s^{2K} + o(s^{2K})$.
Thus, this exhibits a global example with a given value of $K$
as in Theorem \ref{mainthm}.

\subsection{Stable polynomials via row contractions}\label{sec:row}
A natural method for constructing examples of stable polynomials in $\mathbb{B}_d$ mirroring a standard approach in $\mathbb{D}^d$ (see eg. \cite[Section 2.2]{Betal} or \cite{AM05}) goes as follows. Let $A=(A_1,\dots,A_d)$, $A_j\in M_{N,N}(\mathbb{C})$, be a $d$-tuple of matrices giving rise to a row contraction, that is, 
\[
\sum_{j=1}^{d} A_jA_j^* \leq I_N
\]
where $I_N$ denotes the identity matrix, and $A\leq B$ is used to indicate that $B-A$ is a positive semi-definite matrix. 

Set
\[p(z_1,\dots, z_d)=\det\left(I-\sum_{j=1}^{d}z_jA_j\right), \quad z\in \mathbb{B}_d;\]
then $p$ is stable in the ball. 
This follows from the observation
that for $z\in \mathbb{B}_d$ and $v\in \C^N$
\[
\left\|\sum_{j=1}^{d} \bar{z}_j A_j^*v\right\| \leq \sum_{j=1}^{d}|z_j|\|A_j^*v\|
\leq \|z\| \left\langle \sum_{j=1}^{d}  A_jA_j^*v, v \right\rangle^{1/2} 
<\|v\|
\]
by Cauchy-Schwarz.

One might now hope that judicious choices of $A_j$'s would produce examples of stable polynomials with interesting boundary behavior. For instance, one can produce
all $\D^2$-stable polynomials via a block $(N_1+N_2)\times (N_1+N_2)$ contractive matrix $K$ via
the formula
\[
c \det\left(I - K \begin{pmatrix} z_1I_{N_1} & 0 \\
0 & z_2 I_{N_2} \end{pmatrix}\right).
\]
See \cite{Kummert}, \cite{Grinshpan}.

However, the boundary zeros of the stable polynomials that we get via row contractions can
be factored out as degree one polynomials.
This was pointed out to us in personal communication by Kelly Bickel, who also supplied the essence of the following proof. 
\begin{lemma}\label{contrmatdivision}
Suppose $A=(A_1,\dots,A_d)$ is an $N\times N$ row contraction and $A_1^*e_1 = e_1$
for $e_1 = (1,0,\dots, 0)^t$.
Then, $A_j^* e_1 =0$ for $j=2,\dots, d$ and writing
\[
A_1 = \begin{pmatrix} 1 & \vec{0}^t \\
* & \tilde{A}_1 \end{pmatrix},\quad  A_j = \begin{pmatrix} 0 & \vec{0}^t \\
* & \tilde{A}_j \end{pmatrix} \quad j=2,\dots, d
\]
we have
\[
\det\left(I - \sum_{j=1}^{d} z_j A_j\right) = (1-z_1) \det\left(I-\sum_{j=1}^{d} z_j \tilde{A}_j\right).
\]
Note in particular that $(\tilde{A}_1,\dots, \tilde{A}_d)$ is an $(N-1)\times (N-1)$ row contraction.
\end{lemma}
\begin{proof}
Since the tuple $(A_1,\dots,A_d)$ is a row contraction,
and $e_1$ is an eigenvector of $A_1^*$, we obtain that
\[
\|e_1\|^2\geq \langle A_1A_1^*e_1, e_1\rangle+\sum_{j=2}^{d} \langle A_jA_j^*e_1,e_1 \rangle =\|e_1\|^2+\sum_{j=2}^{d}\|A_j^*e_1\|^2.
\]
Hence $A_j^*e_1 = 0$ for $j=2,\dots, d$. 
The determinant formula follows evaluating along the first row.
\end{proof}
As a consequence of Lemma \ref{contrmatdivision}, we obtain that stable polynomials obtained from row contractions factor.

\begin{prop}\label{factordet}
Let $p=\det(I-\sum_{j=1}^{d}z_jA_j)$ where $A=(A_1,\dots, A_d)$ is a row contraction 
with $A_j \in M_N(\mathbb{C})$. 
Then, $p$ has finitely many zeros on $\partial\mathbb{B}_d$,
and we can factor

\[
p(z)=q(z)\prod_{\zeta \in Z_p\cap \partial\mathbb{B}_d}(1-\langle z, \zeta\rangle).
\]
where $q\in \mathbb{C}[z_1,z_2]$ satisfies $Z_q\cap \overline{\mathbb{B}^2}=\emptyset$ 
and the product is over $Z_p\cap \partial \mathbb{B}_d$ (with repetitions to account for multiplicity).
\end{prop}

\begin{proof}
We will show that any zero $\zeta$ on $\partial \mathbb{B}_d$
leads to a factor $1-\langle z ,\zeta\rangle$.
The number of such factors (and the number
of such zeros taken with multiplicity)
can therefore be at most the total degree of $p$.

Suppose $a = (a_1,\dots, a_d) \in \partial \mathbb{B}_d$
and $p(a) = 0$.
Let $U$ be a $d\times d$ unitary matrix with $a^t$ as its first column.
We shall view $U\otimes I_N$ as a block $(dN\times dN)$ (unitary) matrix
with first block column
\[
\begin{pmatrix} a_1 I_N \\ \vdots \\a_d I_N \end{pmatrix}.
\]
Consider the row contraction
\[
(A_1,\dots, A_d) (U\otimes I_N) =: (B_1,\dots, B_d) = B
\]
where $B_1 = \sum_{j=1}^{d} a_j A_j$.
Since $0=p(a) = \det(I-B_1)$, we can conjugate
by a unitary and assume
$B_1^* e_1 = e_1$.
Applying Lemma \ref{contrmatdivision} to $B$,
we have
\[
\det \left(I - \sum_{j=1}^{d} z_j B_j\right) = (1-z_1)\det \left(I-\sum_{j=1}^{d} z_j \tilde{B}_j\right)
\]
with notation from the statement of the lemma.
Writing $B(z\otimes I_N)$ for $\sum_{j=1}^{d} z_j B_j$,
we have $B(U^*z \otimes I_N) = A(z\otimes I_N) = \sum_{j=1}^{d} z_j A_j$.
Note that the first component of $U^*z$ is $\langle z,a\rangle$.
Then, substituting $U^*z$ for $z$ we have
\[
\det(I - \sum_{j=1}^{d}z_j A_j) = (1-\langle z,a\rangle)\det(I - \tilde{B}(U^*z\otimes I_{N-1})).
\]
In the last expression $\tilde{B}(U^*z \otimes I_{N-1}) = \tilde{B}(U^*\otimes I_{N-1}) (z\otimes I_{N-1})$. Setting $\tilde{A} = \tilde{B}(U^*\otimes I_{N-1})$, we have
factored
\[
p(z) = \det\left(I - \sum_{j=1}^{d}z_j A_j\right) = (1-\langle z,a\rangle)\det\left(I - \sum_{j=1}^{d}z_j \tilde{A}_j\right).
\]
With a standard inductive argument we obtain the theorem where
$q(z)$ is a determinantal polynomial $\det(I - \sum_{j=1}^{d}z_j C_j)$
where $C = (C_1,\dots, C_d)$ is a row contraction and $q$ has no
zeros on $\partial \mathbb{B}_d$.

\end{proof}

\begin{remark} \label{remark:proper}
It is also worth pointing out another place where stable
polynomials arise---and in this situation we again have
the property that the relevant stable polynomials
have no boundary zeros.  The setting is 
\emph{rational proper maps}; that is rational
functions $q/p: \mathbb{B}_n \to \mathbb{B}_N$
which are proper maps from one ball $\mathbb{B}_n$
to another $\mathbb{B}_N$.  
Necessarily, the valid denominators $p$ must
be non-vanishing in $\mathbb{B}_n$.
Properness translates to the property that
\[
|p(z)|^2 = \|q(z)\|^2
\]
for $|z| = 1$; here we emphasize that $q$ is a vector
polynomial.
Assuming $q, p$ have no common factors, \cite{CimaSuffridge}
proves $p$ has no zeros on the boundary.
See also Theorem 6.8 of \cite{dA2} (page 207).   $\spadesuit$
\end{remark}

\subsection{Examples beyond the ball setting}
The structure of an ideal of admissible numerators depend heavily on both the choice of domains as well as the given denominator polynomial. We point the way towards future research by examining what happens when $p=1-z_2$ is kept fixed but the underlying domain varies. Presumably, a systematic approach might involve analyzing analogs of a $G_{\phi}$ function adapted to the domain in question.
\begin{example}
For $\alpha >0$, consider the complex ellipsoid
\[\Omega=\mathcal{E}_{\alpha}=\{(z_1,z_2)\in \mathbb{C}^2\colon |z_1|^2+|z_2|^{2\alpha}<1\}.\]
The $\mathcal{E}_{\alpha}$ are examples of pseudoconvex domains in $\mathbb{C}^2$ that fail to be strongly pseudoconvex at $(1,0)$ when $\alpha>1$ (see eg. \cite{dA}). 

We again take $p(z)=1-z_2$. The only zero of $p$ in $\mathcal{E}_{\alpha}$ occurs at the special boundary point $(1,0)$. Clearly, $1-z_2\in \mathcal{I}^{\infty}_p$. 
Now consider $f_N=z_1^N(1-z_2)^{-1}$, $N\in \mathbb{N}$. Taking $0<a<1$, and evaluating $f_N$ along the curve \[\gamma =\{(a(1-r^2)^{1/2\alpha},r)\colon r\in [0,1)\}\subset \mathcal{E}_{\alpha},\] 
we obtain
\[|f_N(a(1-r^2)^{1/2\alpha},r)|^2=a^{2\alpha}\frac{(1-r^2)^{\frac{N}{\alpha}}}{(1-r)^2}=a^{2\alpha}(1+r)^{\frac{N}{\alpha}}(1-r)^{\frac{N}{\alpha}-2},\]
and thus $f_N$ remains bounded near $(0,1)$ if and only if $N\geq 2\alpha$. The fact that $|z_1|^{2\alpha}<1-|z_2|^2$ in $\mathcal{E}_{\alpha}$ can then be used to show that $z_1^{\lceil 2\alpha\rceil+k}/(1-z_2)$ is a bounded rational function in $\mathcal{E}_{\alpha}$ for each integer $k\geq 0$. Thus $\mathcal{I}^{\infty}_{p}=(z_1^{\lceil 2\alpha \rceil}, 1-z_2)$. (We only need the ceiling function if we want to consider globally defined rational functions.)

For instance, when $\alpha=2$, the monomials $z_1,z_1^2, z_1^3\notin \mathcal{I}^{\infty}_p$, and $\mathcal{I}^{\infty}_p=(z_1^4, 1-z_2)$. On the other hand, when $\alpha=1/2$, we find that $(z_1, 1-z_2)=\mathcal{I}^{\infty}_p$, meaning that 
$q(0,1)=0$ becomes both necessary and sufficient for $q/p$ to be a bounded rational function in $\mathcal{E}_{1/2}$. $\diamondsuit$
\end{example}

\section{Constructing polynomials with local behavior} \label{sec:global}

How do we construct polynomials with
prescribed local boundary behavior?
How do we build irreducible polynomials 
$p(z,w) \in \C[z,w]$
with no zeros on $U_2$ but with a real one dimensional curve
of zeros on the boundary?

As discussed in the introduction, we prove the following
theorem and construct several related examples.
Recall that we are interested in parametrizations
of the form
\[
s \mapsto \gamma(s) = \left(cs^Me^{M L(s)}, is^{2M} +i2M \int_{[0,s]} x^{2M} L'(x)dx\right)
\] 
where $L(s) \in \C\{s\}$, $L(0)=0$, $|c|=1$, $M \in \mathbb{N}$ with additional details 
laid out in Theorem \ref{mainthm} depending
on whether we are in isolated type or curve type cases.

\algthm*

Here are some instructive examples.

\begin{example} \label{revNext}
We can reverse engineer the isolated type polynomial in Example \ref{Nextsimplest} 
\[
P_{1/2}(z,w) = w-iw^2-iz^2
\]
which had parametrization 
\[
t \mapsto (t, \phi(t))
\]
where $\phi(t) = (i/2)((1+4t^2)^{1/2}-1)$.
Following Remark \ref{Remark:buildL}, 
we factor $(-i/2)t\phi'(t) = \Phi(t)^2$
and get 
\[
\Phi(t) = \frac{t}{(1+4t^2)^{1/4}}
\]
and $\Psi(s)=  \Phi^{-1}(s)$ 
\[
\Psi(s) = s ((1+4s^4)^{1/2}+2s^2)^{1/2}
\]
so that finally
\[
e^{L(s)} = ((1+4s^4)^{1/2}+2s^2)^{1/2}.
\]
As one can check, this algebraic function has no zeros or
poles in $\C_{*}$. 
We have
\[
L'(s) = \frac{2s}{(1+4s^4)^{1/2}} \text{ and } L'(1/s) = \frac{s}{(1+s^4/4)^{1/2}}
\]
and the last expression evidently has no $s^3$ term in its power series.
$\diamondsuit$
\end{example}

\begin{example} \label{ex:algsimple}
(Example with good local behavior but not global)
    Consider $f(s) = \frac{1}{2}(1+(1+is)^{1/2})$ as a candidate for $e^{L(s)}$.
    One of the branches has $f(0)=1$ while the other branch vanishes
    at $0$. This is the only zero for any analytic continuation.
    The logarithmic derivative is
    \[
    f'(s)/f(s) = \frac{i}{2} \frac{1}{(1+is)^{1/2}(1+(1+is)^{1/2})}
    =
    \frac{1}{2s}\left(1-\frac{1}{(1+is)^{1/2}}\right).
    \]
One branch near $0$ begins $L'(s) = \frac{i}{4} + \frac{3}{16}s + O(s^2)$
    which means $L(s) = (i/4)s + (3/32)s^2 + O(s^3)$
    which is of the form required for isolated type
    in Theorem \ref{mainthm}.  The other branch of $L'(s)$ has
    a simple pole at $0$.
    The Puiseux-Laurent series of $L'(1/s) = 2s(1- s^{1/2}(s+i)^{-1/2})$ at
    $0$ has no $s^3$ term as required.
    
    The parametrization $\gamma(s)$ can
    be explicitly computed (and one could 
    find an explicit polynomial that vanishes
    on the range of $\gamma$).

    Indeed, 
    \[
    \gamma(s) = \left( \frac{s}{2}(1+(1+is)^{1/2}), 
    \frac{3is^2}{2} - 2i \left[ \frac{(1+is)^{3/2}-1}{3} - ((1+
    is)^{1/2}-1)\right]\right).
    \]
   
    The curve does enter $U_2$ though.  
    For instance, for $s= -4-2i$
    \[
    \gamma(s) = (3-i, -56/3 + 14i)
    \]
    using $(1+is)^{1/2} = -2+i$
    and $|3-i|^2=10 < \Im(-56/3 + 14i) = 14$
    shows $\gamma(s) \in U_2$. $\diamondsuit$
\end{example}

\begin{example}
(A non-example---logarithmic singularities in the plane)
Consider $L(s) = iA(s) = \log\left(\frac{1+is}{1-is}\right)$
which has imaginary coefficients because $\left|\frac{1+ir}{1-ir}\right| = 1$ for $r$ real. 
We can compute the integral in our parametrization  
\[
\int_{0}^{s} x^2 A'(x)dx =\int_{0}^{s} \frac{2x^2}{1+x^2} dx
=
2s - 2\tan^{-1}(s).
\]
Note that the zero and pole of $e^{L(s)}$ causes 
logarithmic singularities in the above integral at $\pm i$.
We obtain the parametrization
\[
s \mapsto \left( s \frac{1+is}{1-is} , is^2 +4\tan^{-1}(s) - 4s\right)
\]
which is \emph{not} algebraic. $\diamondsuit$
\end{example}

\begin{example} \label{ex:resatinf}
(Another non-example---non-trivial residue/period at infinity)
Consider
\[
f(s) = (1-s^2)^{1/2} + is
\]
as a candidate for $e^{L(s)}$
as it has no zeroes/poles in $\C$.
The logarithmic derivative is
\[
L'(s) = i(1-s^2)^{-1/2}
\]
and the power series $L'(1/s) = is(1-s^2)^{-1/2}$ \emph{does}
have a $s^3$ term.
Indeed, we see
\[
\int \frac{s^2}{(1-s^2)^{1/2}} ds
\]
is not algebraic as it is essentially
an inverse trig function $\diamondsuit$    
\end{example}

\begin{example}
(Non-example due to non-local periods) 
Consider $f(s) = \frac{1}{2}(1+(1+is^3)^{1/2})$
    as a candidate for $e^{L(s)}$.
    Evidently, it only has a zero at $0$ for one of the
    branches.  The logarithmic derivative
    \[
    L'(s) = \frac{3}{2s}\left(1- \frac{1}{(1+is^3)^{1/2}}\right)
    \]
    and $L'(1/s) = \frac{3s}{2}\left(1 - \frac{s^{3/2}}{(s^3+i)^{1/2}}\right)$
    which has no $s^3$ term.  
    In addition, while one branch of $L'$ has a pole at $0$
    the other leads to analytic $L(s) = \frac{1}{4} is^3 + \frac{3}{32} s^6 + o(s^6)$
    which satisfies the isolated type conditions of Theorem \ref{mainthm}.
Thus, $f$ satisfies all of the local constraints.
However, the relevant integral $\int s^2 L'(s)ds$ upon removing 
algebraic terms contains the term
\[
\int \frac{s}{(1+is^3)^{1/2}} ds
\]
which has a non-zero period when
integrating around two branch points.  $\diamondsuit$
\end{example}

\begin{example} \label{ex:exotic}
(No zeros in $U_2$ but with a curve of zeros on
the boundary.)
Consider $f(s) = ((1-4s^4)^{1/2} + 2is^2)^{1/2}$
which is a modification of Examples \ref{revNext}, \ref{ex:resatinf}.
Notice that $f(s)$ is unimodular for $s$ real 
(for any choices of branch)
which leads to curve type local behavior.
The logarithmic derivative is
\[
\frac{2is}{(1-4s^4)^{1/2}}
\]
and the integral 
\[
\int s^{2} \frac{2is}{(1-4s^4)^{1/2}} ds
\]
is evidently algebraic.  
The resulting parametrization
\[
s \mapsto \left(sf(s) , is^2 + \frac{1}{2}((1-4s^2)^{1/2} - 1)\right) = \left(sf(s), \frac{1}{2}(f(s)^2-1)\right)
\]
gives the zero set
of the polynomial $p(z,w) = w^2+w - iz^2$,
a simple modification of Example \ref{revNext}.
There is no harm in rotating $z$ to obtain
the simpler polynomial $w^2+w-z^2$.
This polynomial has no zeros in $U_2$.  
Indeed, the zero set is better parametrized 
via $w\mapsto (\sqrt{(w^2+w)}, w)$
and we need to show 
\[
|w^2+w| \geq \Im w.
\]
In fact, the stronger inequality $|w^2+w|^2 \geq (\Im w)^2$
holds because it can be rearranged into
\[
(|w|^2 + \Re w)^2 \geq 0.
\]
Equality in this inequality holds on the circle $|w+1/2| = 1/2$
but in our original inequality we need $\Im w \geq 0$.
So, equality holds in $|w^2+w| \geq \Im w$
for $|w+1/2|=1/2$ and $\Im w \geq 0$.  
This is how the parametrization from the introduction
and Figure \ref{fig:Bernoulli} is obtained. $\diamondsuit$
\end{example}

\begin{example}
How might we construct examples of polynomials
with local behavior in Theorem \ref{mainthm}
with $M=2$? 
The simplest (non-global) choice of $L(s)$ with $M=2$
from Example \ref{M2ex} was $\tilde{L}(s) = is^2+s^4+s^5$.
Recall the $s^5$ term is present to force
the parametrization to be irreducible, i.e. not
a function of $s^2$.
Since $e^{\tilde{L}} + e^{-\tilde{L}} = 2 \cosh(\tilde{L})$,
we can approximate the right-hand side with a
Taylor polynomial and analyze the algebraic
function satisfying
\[
f(s) + f(s)^{-1} = 2 P(s)
\]
where $P(s) = 1 - \frac{1}{2}s^4 +is^6+is^7
= \cosh(\tilde{L}(s)) + O(s^8)$.
Thus, $f(s)$ satisfies $y^2 - 2P(s) y + 1 =0$
which means
\[
f_{\pm}(s) = P(s) \pm \sqrt{P(s)^2-1}.
\]
Analytic continuations of $f(s)$ do not vanish in
$\C$ (as $f(s) = 0$ implies $P^2=P^2-1$).

Now, note that
\[
\sqrt{(\cosh(\tilde{L})+O(s^8))^2 - 1} = \sqrt{\sinh^2(\tilde{L}) + O(s^8)} = \sinh(\tilde{L})(1+O(s^6)) = \sinh(\tilde{L}) + O(s^8)
\]
so that
\[
f = P \pm \sqrt{P^2-1} = \cosh(\tilde{L}) \pm \sinh(\tilde{L}) + O(s^8) = e^{\pm \tilde{L}} + O(s^8)
\]
and that
\[
f'/f = \pm \tilde{L}' + O(s^7).
\]
Thus, taking the local log of $f_{+}$, $L = \log f_{+}$,
we have $L = \tilde{L} + O(s^8)$
which has the correct local expansion in order to
yield a local parametrization $\gamma$ that avoids $U_2$.
However, the other branch begins with $-\tilde{L}$
and therefore does not avoid $U_2$. $\diamondsuit$
\end{example}

We conclude the paper with the proof of Theorem \ref{algthm}.

\begin{proof}[Proof of Theorem \ref{algthm}]
To start
\[
0 \equiv p(\gamma(s))
\]
and taking the derivative yields
\begin{align*}
0&=(cMs^{M-1}e^{ML} + cs^Me^{ML}ML') p_1(\gamma(s)) + (2Mis^{2M-1}+2Mis^{2M}L')p_2(\gamma(s)) \\
&= Ms^{M-1}(1+sL')(ce^{ML} p_1 + 2is^{M} p_2)
\end{align*}
where $p_1,p_2$ denote partial derivatives.
This implies
\[
c s^{M} e^{ML} p_1 + 2is^{2M} p_2 = 0
\]
on $\gamma$.  
Consider then the polynomial in $z,w,t$
\[
Q(z,w,s) = z p_1(z,w) +2it p_2(z,w)
\]
which satisfies $Q(\gamma(s), s^{2M}) \equiv 0$.

We can assume without loss of generality that $c=1$.
Let $R(z,t)$ denote the resultant
with respect to $w$ of $p(z,w)$ and 
$Q(z,w,t)  \in \C[z,w,t]$.
Then, there exist polynomials 
$B(z,w,t),C(z,w,t)$ such that
\[
R(z,t) = B(z,w,t) p(z,w) + C(z,w,t)Q(z,w,t).
\]
The degree of $B(z,w,t)$ in $w$ is less than
the degree of $Q(z,w,t)$ in $w$
and similarly
the degree of $C(z,w,t)$ in $w$ is less than the degree of 
$p(z,w)$ in $w$. See for instance \cite[Proposition 4.19]{Basu}.

If $R(z,t)$ is identically zero, 
then we get a contradiction, as follows.
We would have
\[
-B(z,w,t) p(z,w) = C(z,w,t)Q(z,w,t)
\]
and since $p$ is irreducible and the degree 
of $C$ in $w$ is less than that of $p$,
we have that $p(z,w)$ is a multiple of $Q(z,w,t)$.
As $p$ does not involve $t$ we must have $p_2 = 0$ but then
$p$ cannot be a multiple of $p_1$.  
So, $R(z,t)$ is not identically zero.

Since $0 \equiv p(\gamma(s)) \equiv Q(\gamma(s),s^{2M})$, we have
$R(s^M e^{M L(s)}, s^{2M}) \equiv 0$ 
for $s$ near $0$.  Writing
\[
R(z,t) = \sum_{j,k=0}^{D} R_{j,k} z^j t^k
\]
we have
\[
0 = \sum_{j,k=0}^{D} R_{j,k} s^{Mj+2Mk} e^{MjL(s)}
\]
We see that $f(s) := e^{L(s)}$ satisfies
a nontrivial polynomial in $\C(s)[y]$,
the polynomials with coefficients in the field 
$\C(s)$ of rational functions.
In other words, $f(s) = e^{L(s)}$ is algebraic
and hence has analytic continuation through
the Riemann sphere minus finitely many points.
Necessarily, $f'(s)$ and $f'(s)/f(s)$ are
both algebraic and belong to $\C(s,f(s))$.  
In particular, $f'(s)/f(s)$ has singularities that are at worst simple poles 
(at zeros, poles, or algebraic zeros/poles of $f(s)$)
or branch points.

Next, consider
\[
g(s) := is^{2M} +i2M \int_{\Gamma_s} x^{2M} f'(x)/f(x)dx
=
2Mi\int_{\Gamma_s} x^{2M-1}(1+xf'(x)/f(x))dx.
\]
taken along some path $\Gamma_s$ from $0$ to $s$.
Since 
\[
p\left(s^M f(s)^M, g(s)\right) =0,
\]
$g(s)$ satisfies a polynomial with coefficients 
in $\C(s,f(s))$.
We claim that $g \in \C(s,f(s))$.
If not, $g$ satisfies a
(monic) minimal
polynomial $Q(y)$ with coefficients in 
$\C(s,f(s))$ and degree greater than $1$.
We can write
\[
Q(y) = Q(s;y) = \sum_{j=0}^{K} Q_j(s) y^j
\]
($Q_K=1$)
and
differentiating
\begin{equation}\label{diffeq}
0 = s\frac{d}{ds} Q(s;g(s)) 
= \sum_{j=0}^{K} \left( sQ_j'(s)g^j
+ Q_j(s)j g^{j-1} 2Mis^{2M}(1+sf'(s)/f(s))
\right).
\end{equation}
Since $Q_j(s) \in \C(s, f(s))$, 
we have $Q_j'(s) \in \C(s, f(s))$.
Since $Q_K'(s) = 0$, \eqref{diffeq} yields a polynomial
in $y$ with degree lower than $Q$,
contradicting minimality.  
Therefore, $g \in \C(s,f(s))$.
Notice that if the analytic continuation of $f(s)$
had a zero or pole at a point $s_0\in \C$ other than $0$,
then $g(s)$ would have a logarithmic 
singularity at $s_0\in \C$ contrary
to the fact that $g(s) \in \C(s,f(s))$.
Thus, we conclude that the analytic continuation
of $f(s)$ has no zeros in $\C_*= \C\setminus\{0\}$.

To analyze the integral $\int x^{2M} f'(x)/f(x)dx$
at infinity, we examine the Puiseux-Laurent series
of $f(1/s) = \sum_{j=m}^{\infty} a_j s^{j/k}$; here $a_m \ne 0$.
Using the change of variables $x= u^{-k}$ we have
\[
\int u^{-2kM} \frac{f'(1/u^k)}{f(1/u^k)}(-k)u^{-k-1}du.
\]
This will have a logarithmic singularity if
\[
u^{-2k(M+1)} \frac{f'(1/u^k)}{f(1/u^k)}u^{-k-1}
\]
has nonzero residue at $u=0$
or in other words if $f'(1/u^k)/f(1/u^k)$ contains
the term $u^{k(2M+1)}$.
This is equivalent to the Puiseux-Laurent series
of $f'(1/s)/f(1/s)$ containing the term $s^{2M+1}$
as claimed.

\end{proof}

\subsection*{Acknowledgements}

We thank Kelly Bickel for numerous conversations, comments, and suggestions in the early stages of this work.

\end{document}